\documentclass[a4paper,11pt]{article}
\usepackage[utf8]{inputenc}
\usepackage[T1]{fontenc}
\usepackage[french,english]{babel}
\usepackage{amssymb}
\usepackage{mathrsfs}
\usepackage{amsthm}
\usepackage{amsmath}
\usepackage{longtable}
\usepackage{caption}
\usepackage{relsize}
\usepackage{xcolor}
\usepackage{enumitem}
\usepackage[document]{ragged2e}
\usepackage{hyperref}
\usepackage{etoolbox}
\usepackage{geometry}
\usepackage{graphicx}
\usepackage{xspace}
\usepackage{caption}
\usepackage{subcaption}
\usepackage{bm}
\usepackage{bbm}
\usepackage[capitalize]{cleveref}
\usepackage{thmtools}
\usepackage{cases}
\usepackage{array}
\usepackage[title]{appendix}

\hypersetup{
    colorlinks=true,
    linkcolor=[rgb]{0.2, 0, 0.7},
    filecolor=magenta,
    urlcolor=cyan,
    bookmarks=true,
    citecolor=[rgb]{0,0.4,0},
}

\crefformat{equation}{#2\normalfont{{(\textup{#1})}}#3}

\usepackage[backend=bibtex,
            giveninits=true,
            style=numeric,
            doi=false,
            isbn=false,
            url=false,
            eprint=false,
            maxbibnames=99
            ]{biblatex}
\newcommand{\RR}{\mathbb{R}}
\newcommand{\NN}{\mathbb{N}}

\newcommand{\EE}{\mathbb{E}}

\newcommand{\cA}{\mathcal{A}}
\newcommand{\cC}{\mathscr{C}}
\newcommand{\cG}{\mathcal{G}}

\newcommand{\cP}{\mathcal{P}}

\newcommand{\cM}{\mathcal{M}}

\newcommand{\cV}{\mathcal{V}}
\newcommand{\cE}{\mathcal{E}}

\newcommand{\cB}{\mathcal{B}}

\newcommand{\mf}{\mathfrak{m}}
\newcommand{\gothA}{\mathfrak{A}}

\newcommand{\be}{{\bm{e}}}

\newcommand{\bgamma}{\bm{\gamma}}

\DeclareSymbolFont{matha}{OML}{txmi}{m}{it}% txfonts
\DeclareMathSymbol{\varv}{\mathord}{matha}{118}

\newcommand{\vertiii}[1]{{\left\vert\kern-0.25ex\left\vert\kern-0.25ex\left\vert #1
    \right\vert\kern-0.25ex\right\vert\kern-0.25ex\right\vert}}
\newcommand{\norm}[2]{ \left \lVert {#1}  \right \rVert_{#2}}
\newcommand{\module}[1]{\left \lvert {#1} \right \rvert}

\DeclareMathOperator{\Lip}{Lip}
\DeclareMathOperator{\diam}{diam}

\DeclareMathOperator{\leb}{\mathscr{L}}

\DeclareMathOperator{\argmax}{argmax}

\DeclareMathOperator{\sgn}{sgn}

\newcommand{\sob}[2]{W^{{#1}, {#2}}({\Gamma})}
\newcommand{\csob}[2]{V^{{#1}, {#2}}({\Gamma})}
\newcommand{\bsob}[3]{V_{#3}^{{#1}, {#2}}({\Gamma})}

\newcommand{\cssob}{V^{{1}, {2}}({\Gamma})}
\newcommand{\bssob}{V_{\bgamma}^{{1}, {2}}({\Gamma})}

\newcommand{\Lx}[1]{L^{#1}(\Gamma)}

\DeclareMathOperator{\dw}{{\bf d}_1}
\newcommand{\half}{\frac 1 2}

\renewcommand{\eqref}[1]{\cref{#1}}

\theoremstyle{plain}
\newtheorem{thm}{Theorem}[section]

\newtheorem{lem}[thm]{Lemma}
\newtheorem{prop}[thm]{Proposition}

\theoremstyle{definition}
\newtheorem{defi}[thm]{Definition}

\theoremstyle{remark}
\newtheorem{rem}[thm]{Remark}
\newtheorem{ex}[thm]{Example}

\crefname{thm}{Theorem}{Theorems}
\crefname{cor}{Corollary}{Corollaries}
\crefname{lem}{Lemma}{Lemmata}
\crefname{prop}{Proposition}{Propositions}
\crefname{def}{Definition}{Definitions}
\crefname{rem}{Remark}{Remarks}

\newcounter{assumptions}
\title{Stationary Mean Field Games on networks with sticky transition conditions\footnote{Partially supported by  INdAM-GNAMPA and PRIN-PNRR P20225SP98 ``Some mathematical approaches to climate change and its impacts''.  J.B. was partially supported by the ANR (Agence Nationale de la Recherche) through the
COSS project ANR-22-CE40-0010 and the Centre Henri Lebesgue ANR-11-LABX-0020-01 and by Rennes Métropole through the Collège doctoral de Bretagne. This work was conducted while J.B. was visiting  Sapienza Università di Roma. He
wishes to thank the university for its hospitality.}}
\author{Jules Berry\footnote{Univ Rennes, INSA, CNRS, IRMAR - UMR 6625, Rennes F-35000, France ({\fontfamily{cmtt}\selectfont jules.berry@insa-rennes.fr}).} \and Fabio Camilli\footnote{Dip. di Ingegneria e Geologia, Universit\`a degli Studi ``G. d'Annunzio" Chieti-Pescara, viale Pindaro 42, 65127 Pescara (Italy) ({\fontfamily{cmtt}\selectfont fabio.camilli@unich.it}).}}

\date{}

\begin{document}
\maketitle

\begin{abstract}
We study stochastic Mean Field Games on networks with sticky transition conditions. In this setting, the diffusion process governing the agent's dynamics can spend finite time both in the interior of the edges and at the vertices. The corresponding generator is subject to limitations concerning second-order derivatives  and the invariant measure breaks down into a combination of an absolutely continuous measure within the edges and a sum of Dirac measures positioned at the vertices. Additionally, the value function, solution to the Hamilton-Jacobi-Bellman equation, satisfies generalized Kirchhoff conditions at the vertices.
\end{abstract}



\justify
\section{Introduction}
The theory of Mean Field Games (MFG in short), introduced in  \cite{HCM,LL2007}, considers differential games as the number of agents approaches infinity.  The adaptation of the MFG theory   to network  has been considered  in \cite{ADLT2019,ADLT2020, CM2024, CM2016,}. It revolves around a PDE system comprising two differential equations: a Hamilton-Jacobi-Bellman equation and a Fokker-Planck equation, each defined within the network's edges. These equations are complemented with appropriate boundary conditions, as well as initial-final conditions for evolutive problems, and transition conditions at the vertices. The  transition conditions are pivotal both in modeling, describing how agents behave upon reaching a vertex, and theoretically, ensuring the problem possesses properties  that allow it to be studied, for example the Maximum Principle for the Hamilton-Jacobi-Bellman equation.\par
In previous research on MFG on networks, the vertex transition condition for the Hamilton-Jacobi-Bellman equation typically adheres to the classical Kirchhoff condition. This condition implies that  the  agent spend  zero time at the vertices and enter  one of the adjacent edges with a probability determined by specific coefficients, as dictated by the associated Markov process. In terms of duality, the vertex transition condition for the Fokker-Planck equation aligns with flux conservation principles. These conditions collectively indicate that the distribution of agents remains absolutely continuous, with no mass concentration occurring at the vertices. Nevertheless, in various physical models implemented on networks, such as those dealing with traffic flow or data transmission, congestion at vertices is a common occurrence. This congestion arises from the delay in distributing agents, whether they be vehicles or data packets, along the corresponding incident edges.  Consequently, it becomes important to develop a mathematical framework that can capture and describe these congestion phenomena.\par
In the unidimensional setting, sticky diffusion processes were first considered by Feller \cite{F1952,F1954} through the semigroup approach by adding a second order term in the boundary condition defining the domain of the infinitesimal generator. The name sticky is motivated by the fact that these processes have a strictly positive occupation time on the boundary. A fundamental example is the sticky (reflected) Brownian motion. It was proved in \cite{IM1963} that it can be obtained by slowing down a reflected Brownian motion. In addition, a characterization in terms of a stochastic differential equation was obtained in \cite{EP2014,B2014}. These results were then extended to general diffusion processes in \cite{SS2017}. Sticky processes have for instance found applications to operations research \cite{HL1981}, finance \cite{KKR2007} and epidemics models \cite{CF2012}. Recently, their analogues on networks have been studied in \cite{KPS2012a,KPS2012b,KPS2012c,BD2024,BC2024}.\par
 In our investigation, we explore MFG  on networks governed by such sticky diffusion processes. Unlike the nonsticky diffusion processes governed by Kirchhoff  conditions, stickiness leads to an accumulation of agents at the vertices and their distribution across the network splits into two distinct components. One component is absolutely continuous along the arcs, while the other concentrates at the vertices. We demonstrate that the stationary distribution of the sticky process can be characterized similarly to the problem discussed in \cite{ADLT2019} for the absolute continuous component inside the edges, while the singular part is  represented as a sum of Dirac masses at the vertices with coefficients proportional to the stickiness of the process. A similar decomposition has also been observed  in \cite{BOURABE2019,SANTA2024,ACG2021}.

Once the characterization of the stationary distribution is obtained,   we turn to the study the well-posedness of Hamilton-Jacobi equations associated to discounted infinite horizon and long-term averaged optimal control problems of the sticky diffusion process.
In order to identify the appropriate transmission condition to impose at vertices, we establish a verification theorem for the  Hamilton-Jacobi equation associated with a discounted infinite-horizon optimal control problem with bounded controls, utilizing the Itô formula as described in \cite{SS2017,BC2024}. The correct condition turns out to be a generalized Kirchhoff condition incorporating the stickiness coefficient and the cost incurred at the vertices. Differential equations on networks with generalized transition conditions at the vertices have been also studied in \cite{vBN1996,MR2007,O2021}.

 Subsequently, we demonstrate the existence of solutions for the stationary Mean Field Game system through a fixed point argument, supplemented by establishing uniqueness leveraging a classical monotonicity assumption on the cost. In examining the Mean Field Games system, a pivotal aspect is the duality between the linearized Hamilton-Jacobi-Bellman equation and the Fokker-Planck equation, which is satisfied also in this case.\par
The paper is structured as follows. Section \ref{sec:2} provides basic definitions for functions on networks and related functional spaces. Section \ref{sec:3} examines diffusion processes on networks and the Fokker-Planck equation for the corresponding stationary distribution. In Section \ref{sec:4}, we focus on the Hamilton-Jacobi-Bellman   equation associated with discount and ergodic control problems. Lastly, Section \ref{sec:5} presents the proofs of the main results concerning the existence and uniqueness for the Mean Field Game system.
%%%%%%%%%%%%%%%%%
%               %
%%%%%%%%%%%%%%%%%
\section{Networks and functions spaces}\label{sec:2}
This section presents the basic notions and properties of a network  and    some function  spaces associated to it    \cite{ADLT2019,ADLT2020,CM2016}.
We consider a finite set $\cV$ of points in $\RR^d$ which will be the set of vertices of the network. We   then consider a subset $P \subset \cV \times \cV$ satisfying
\begin{itemize}
 \item if $\varv \in \cV$ then $(\varv,\varv) \notin P$,
 \item for every $\varv_1 \in \cV$, there exists a $\varv_2 \in \cV$ such that either $(\varv_1, \varv_2) \in P$ or $(\varv_2, \varv_1) \in P$.
\end{itemize}
We say that $\varv$ is a boundary vertex if there a unique $\bar \varv\in \cV$ for which either $(\varv, \bar \varv)$ or $(\bar\varv,  \varv)$ belongs to $P$ and we denote with $\partial \cV$ the set of the boundary vertices. 
We define the set of edges of the network $\cE$ as the set of all segments between the two points of a pair in $P$
\begin{equation}\label{eq:edges}
 \cE = \big \{ \{ \theta \varv_1 + (1 - \theta)\varv_2 : \ \theta \in [0,1] \} : \ (\varv_1,\varv_2) \in P \big \}.
\end{equation}
Since $\cE$ is a finite collection of subsets of $\RR^d$, we can index it by a finite set $\cA$ so that
$
 \cE = \{\Gamma_\alpha \subset \RR^d : \alpha \in \cA \},
$
each $\Gamma_\alpha$ being one the the segments defined  in \eqref{eq:edges}. Finally  the network is given by $\Gamma = \cup_{\alpha \in \cA} \Gamma_\alpha$
and we endow it with the topology induced by the geodesic distance on $\Gamma$.   In all of this work we assume that the network $\Gamma$ is connected. \\
For   $\varv \in \cV$, we define $\cA_\varv = \left \{ \alpha \in \cA \ : \varv \in \Gamma_\alpha \right \}$, that is the set of all indices $\alpha \in \cA$ such that the vertex $\varv$ belongs to the edge $\Gamma_\alpha$. We assume that, for each pair $\alpha, \beta \in \cA$  with $\alpha \neq \beta$, one has $\Gamma_\alpha \cap \Gamma_\beta = \{\varv \}$    whenever $\alpha, \beta \in \cA_\varv$ and $\Gamma_\alpha \cap \Gamma_\beta = \varnothing$ otherwise.
Let $\Gamma_\alpha$ be the edge with vertices $\varv_i$ and $\varv_j$. We denote by $L_\alpha \in (0,\infty)$ the length of the  edge and we define a unit vector
 $\be_\alpha = ( \varv_j - \varv_i)/ L_\alpha$.
Then   $\Gamma_\alpha$  admits a parametrization $\pi_\alpha \colon [0, L_\alpha] \to \Gamma_\alpha$ defined by $\pi_\alpha(s) = \varv_i + s \be_\alpha$.\\
To a function $u : \Gamma \to \RR$, we associate    the function $u_\alpha : (0,L_\alpha) \to \RR$, $\alpha \in \cA$, defined by
\[
 u_\alpha (y) =  ( u \circ \pi_\alpha ) (y), \quad \textnormal{for every } y \in (0,L_\alpha).
\]
For $x \in \Gamma_\alpha$, we set
\begin{equation}\label{eq:continuous_extension_network}
 u_{| \Gamma_\alpha}(x) = \begin{cases}
                          u_\alpha \circ \pi_\alpha^{-1}(x) \quad & \textnormal{for } x \in \Gamma_\alpha \setminus \cV, \\
                          \lim_{y \to 0^+} u_\alpha(y) \quad & \textnormal{if } x = \pi_\alpha (0), \\
                          \lim_{y \to L_\alpha^- } u_\alpha(y) \quad & \textnormal{if } x = \pi_\alpha(L_\alpha),
                         \end{cases}
\end{equation}
if the previous limits exist.\\
For every Borel set $ A \in \cB(\Gamma)$, we define the  Lebesgue  measure  of $A$   by
$
 \leb(A) = \sum_{\alpha \in \cA} \pi_\alpha \# \leb_\alpha (A \cap \Gamma_\alpha)
$
where $\leb_\alpha$ is the usual one dimensional Lebesgue measure on $[0,L_\alpha]$. Clearly  $\leb(\Gamma) = \sum_{\alpha \in \cA} L_\alpha$. For a $\leb$-integrable function $f : \Gamma \to \RR$ we then have
\[
 \int_\Gamma f(x)\ dx :=\int_\Gamma f(x) \leb(dx) = \sum_{\alpha \in \cA} \int_{\Gamma_\alpha} f(x)  \pi_\alpha \# \leb_\alpha(dx) = \sum_{\alpha \in \cA} \int_0^{L_\alpha} f_\alpha(y)\ dy.
\]
%%%%%%%%
We denote by $\cC(\Gamma)$ the space of continuous real valued functions on $\Gamma$, which is a Banach space when equipped with the norm $\norm{u}{\cC(\Gamma)} = \sup_{x \in \Gamma} \module{u(x)}$. It is also convenient to allow functions to be discontinuous at the junctions but continuous in each edge. We first define
\[
 \mathcal{PC}(\Gamma) = \left \{ u : \Gamma \to \RR:\, u_\alpha \in \cC(0,L_\alpha)
\textnormal{ and the limits in \eqref{eq:continuous_extension_network} exist}
\textnormal{ for each } \alpha \in \cA \right \}.
\]
 endowed with the norm of uniform convergence on each edge
\[
\norm{u}{\mathcal{PC} (\Gamma)} = \max_{\alpha \in \cA} \norm{u_\alpha}{L^\infty((0,L_\alpha))},
\]
 We then consider the following equivalence relation
\[
\forall u,v \in  \mathcal{PC} (\Gamma) \quad \left ( u \sim v \right ) \Leftrightarrow \left (u_{\alpha}(x) = v_{\alpha}(x), \ \forall x \in (0,L_\alpha),\  \forall \alpha \in \cA \right ),
\]
and finally we define  $PC(\Gamma) =  \mathcal{PC}(\Gamma) / \sim$.
We observe that $PC(\Gamma)$ with the induced norm is a Banach space. We also denote by $PC^{\varsigma}(\Gamma)$, for $\varsigma \in (0,1]$ the subspace of all $u \in PC(\Gamma)$ such that $u_\alpha \in \cC^\varsigma([0,L_\alpha])$ for all $\alpha \in \cA$,   i.e. such that
\[
 \sup_{\substack{x,y \in [0,L_\alpha] \\ x \neq y}} \frac{\module{u_\alpha(x) - u_\alpha(y)}}{\module{x-y}^\varsigma} < \infty \quad \textnormal{for all } \alpha \in \cA.
\]
For   $x\in \Gamma_\alpha$,  $\alpha \in \cA$,  we   define the derivative of $u$ at $x \in \Gamma_\alpha \setminus \cV$ as the directional derivative
\begin{align*}
\partial_\alpha u(x) & = \lim_{h \to 0} \frac{u(x+ h \be_\alpha) - u(x)}{h} = \lim_{h \to 0} \frac{u_\alpha(\pi_\alpha^{-1}(x + h \be_\alpha)) - u_\alpha(\pi_\alpha^{-1}(x))}{h} .
% & = \partial u_\alpha(\pi_\alpha^{-1}(x)) d(\pi_\alpha^{-1})(x) \cdot \be_\alpha
%  = \partial u_\alpha(\pi_\alpha^{-1}(x))
\end{align*}
 Higher order derivatives  are defined in the same way.  At the vertices, we   define $\partial_\alpha u(\varv )$ as the {outward directional derivative} of $u$ at $\varv  \in \cV$ for each $\alpha \in \cA_\varv$, i.e.
\[
 \partial_\alpha u(\varv) = \begin{cases}
                               \lim_{h \downarrow 0} \frac{u_\alpha(0) - u_\alpha(h)}{h} \quad &\textnormal{if } \varv = \pi_\alpha(0), \\
                               \lim_{h \downarrow 0} \frac{u_\alpha(L_\alpha) - u_\alpha(L_\alpha - h)}{h} \quad &\textnormal{if } \varv = \pi_\alpha(L_\alpha),
                              \end{cases}
\]
when the above limits exist. Notice that if we define
\begin{equation}\label{eq:def_n}
 n_{\varv,\alpha} = \begin{cases}
                 1 \quad &\textnormal{if } \varv = \pi_\alpha(L_\alpha), \\
                 -1 \quad &\textnormal{if } \varv = \pi_\alpha(0),
                \end{cases}
\end{equation}
then $\partial_\alpha u(\varv) = n_{\varv,\alpha} \partial u_\alpha(\pi_\alpha^{-1}(\varv))$.\\
For every integer $k \geq 1$, the function space
\[
 \cC^k(\Gamma) = \left \{ u \in \cC(\Gamma) : \ u_\alpha \in \cC^k([0,L_\alpha]), \ \forall \alpha \in \cA \right \},
\]
equipped with the norm
 $
  \norm{u}{\cC^k(\Gamma)} = \sum_{\alpha \in \cA} \sum_{0 \leq j \leq k} \norm{\partial^j u_\alpha}{L^\infty(0,L_\alpha)}
$
is a Banach space. We will also need to consider \emph{H\"older} continuous functions on $\Gamma$. For each $k \in \NN$ and $\varsigma\in (0,1]$  we define
 \begin{equation}\label{eq:def_holder}
  \cC^{k,\varsigma}(\Gamma) = \big \{ u \in \cC^k(\Gamma) : \ \norm{u}{\cC^{k,\varsigma}(\Gamma)} < \infty \big \},
 \end{equation}
where
\[
 \norm{u}{\cC^{k,\varsigma}(\Gamma)} = \norm{u}{\cC^k(\Gamma)} + \max_{\alpha \in \cA} \sup_{\substack{x, y \in [0,L_\alpha] \\ x \neq y}} \frac{\module{\partial^k u_\alpha(x) - \partial^k u_\alpha(y)}}{\module{x - y}^\varsigma}.
\]
For $p \in [1, \infty]$, we denote by $L^p(\Gamma) := L^p(\Gamma, \cB(\Gamma), \leb)$ and we notice that
\[
 L^p(\Gamma) = \left \{ u : \Gamma \to \RR :\ u_\alpha \in L^p(0,L_\alpha) \textnormal{ for every } \alpha \in \cA \right \}.
\]
with the equivalent norm  $\norm{u}{L^p(\Gamma)} = \left ( \sum_{\alpha \in \cA} \norm{u_\alpha}{L^p(0,L_\alpha)}^p \right )^{\frac{1}{p}}$ for $1 \leq p < \infty$ and $\norm{u}{L^\infty(\Gamma)} = \max_{\alpha \in \cA} \norm{u_\alpha}{L^\infty(0,L_\alpha)}$.
%Of course for every $p \in [1,\infty]$ the space $L^p(\Gamma)$ is a Banach space and $L^2(\Gamma)$ is an Hilbert space for the scalar product
%\[
%(u,v)_{L^2(\Gamma)} = \int_\Gamma u(x) v(x)\ dx.
%\]
 For any integer $k \geq 1$ and every $p \in [1,\infty]$, we define the Sobolev space
 \[
  \sob{k}{p} = \big \{ u \in L^p(\Gamma) :\ u_\alpha \in W^{k,p}((0,L_\alpha)),\ \forall \alpha \in \cA \big \}.% \cong \prod_{\alpha \in \cA} W^{k,p}((0,L_\alpha)).
 \]
endowed  with the norm
$ \norm{u}{W^{k,p}} =   ( \norm{u}{L^p(\Gamma)}^p + \sum_{j = 1}^k \norm{\partial^j u}{L^p(\Gamma)}^p   )^{\frac{1}{p}}$  for $1 \leq p < \infty$   and
$ \norm{u}{W^{k,\infty}} = \norm{u}{L^\infty(\Gamma)} + \sum_{j = 1}^k \norm{\partial^j u}{L^\infty(\Gamma)}$.\\
We also consider Sobolev space with prescribed value at the vertices.  For a given set of strictly positive real numbers
$\bgamma = \left \{ \gamma_{\varv,\alpha} \in (0, + \infty) : \varv \in \cV, \alpha \in \cA_\varv\right \}$, we define
 \[
  \bsob{k}{p}{\bgamma} = \big \{ u \in \sob{k}{p} \ : \frac{u_{|\Gamma_\alpha}(\varv)}{\gamma_{\varv,\alpha}} = \frac{u_{|\Gamma_\beta}(\varv)}{\gamma_{\varv,\beta}},\ \forall \varv \in \cV,\ \forall \alpha, \beta \in \cA_\varv \big \}.
 \]
 Moreover if $\gamma_{\varv,\alpha} = 1$ for every $\alpha \in \cA_\varv$ and each $\varv \in \cV$ we simply write $\csob{k}{p}$ and one has the following identification
 $
 \csob{k}{p} = \cC(\Gamma) \cap \sob{k}{p}.
$
These spaces will be alternatively used as the space in which we look for a solution of a given PDE on $\Gamma$ or as a set of test functions. For any $k \geq 1$ and every $p \in [1, \infty]$, the spaces $\sob{k}{p}$ and $\bsob{k}{p}{\bgamma}$ are Banach space and $\sob{k}{2}$ and $\bsob{k}{2}{\bgamma}$ are Hilbert space with obvious inner product.\\
We finally denote by $\cP_1(\Gamma)$ the space of  Borel probability  measures on $\Gamma$ endowed with the Kantorovich-Rubinstein metric $\dw$  (we refer the reader to \cite[Chapter 6]{V2009} for precise definitions).
%%%%%%%%%%%%
%          %
%%%%%%%%%%%%
  \section{Diffusion process on networks with sticky transition conditions}\label{sec:3}
We introduce diffusion processes on networks with sticky transition conditions, or transition conditions with spatial delay, at the vertices. We consider the linear differential operator $\cG$  on $\Gamma$  defined on each edge $\Gamma_\alpha$ by
\begin{equation}\label{eq:s1_eq1}
    \cG_\alpha f(x)=\mu_\alpha \partial^2  f(x)+b_\alpha(x)\partial  f(x)\qquad \textnormal{for all } x \in \Gamma_\alpha \setminus \cV,\, \alpha \in \cA, \, f \in D(\cG),
\end{equation}
with domain
\begin{equation}\label{eq:s1_eq1a}
D(\cG) = \left \{f \in \cC^2(\Gamma) : \,\begin{array}{l}\cG f\in \cC(\Gamma),\,\eta_\varv \cG f(\varv) + \sum_{\alpha \in \cA_\varv} \mu_\alpha \gamma_{\varv,\alpha} \partial_\alpha f(\varv) = 0 \\[4pt] \textnormal{ for all }  \varv \in \cV \setminus \partial \cV,
 \partial_\alpha f(\varv) = 0\,   \textnormal{ for all } \varv \in \partial \cV,\,\alpha\in\cA_\varv \end{array} \right \},
\end{equation}
where $b \in PC(\Gamma)$, $\eta_\varv\ge 0$,  $\mu_\alpha, \gamma_{\varv,\alpha}>0$, $\alpha\in\cA$, and $\sum_{\alpha \in \cA_\varv} \mu_\alpha \gamma_{\varv,\alpha} =1$ for all $\varv\in\cV\setminus \partial \cV$ (note that $\cG$ does not depend on the choice of $b$ in the corresponding equivalence class).
According to \cite[Theorem 3.1]{FW1993}, there exists a Feller process $X$, with continuous paths, defined on $\Gamma$ with generator $\cG$.

This process behaves like a standard unidimensional diffusion while in the interior of the edges and, upon reaching an interior vertex $\varv \in \cV \setminus \partial \cV$, is randomly dispatched to another edge. The latter process involves two parameters determined in the domain \eqref{eq:s1_eq1a}: stickiness $\eta_\varv$ which, roughly speaking, prescribes the average time spent by the process at $\varv$  and the probability to enter a given edge $\alpha \in \cA_\varv$, which is given by the quantity $\mu_\alpha \gamma_{\varv, \alpha}$. Note also that the Neumann boundary condition imposed on exterior vertices in \eqref{eq:s1_eq1a} induces an instantaneous reflection of the process. In full generality, it is also possible to consider sticky behavior on exterior vertices.

The process can be expressed as $X(t) = (x(t), \alpha(t)) \in \mathbb{R}_+ \times \mathcal{A}$, where $ \alpha$ represents the current edge on which $X$ is located, and \( x \) specifies its position along this edge. In the non-sticky case ($\eta_\varv = 0$), it was shown in \cite{FS2000} that the pair $(x, \alpha)$ satisfies a stochastic differential equation with random coefficients. For the sticky case ($ \eta_\varv > 0$), Itô and McKean \cite{IM1963} demonstrated that the process's paths can be derived from those of the non-sticky case through suitable random time changes. For the sticky Brownian motion on networks, we refer to \cite{KPS2012a, KPS2012b, KPS2012c, BD2024}.

For general diffusion processes, let us first describe the setting considered in \cite{SS2017}, where a path-wise characterization of the process associated to $\cG$  has been obtained for the case of a simple graph $\Gamma^0$ composed of one vertex $\cV = \{0\}$ and two semi-infinite edges $\mathcal{E} = \{\Gamma^0_+, \Gamma^0_- \}$,  where $\Gamma^0_+ = [0, + \infty)$ and $\Gamma^0_- = (-\infty, 0]$.  In this case, we can rewrite  the  domain  of $\cG$ as
  \begin{equation}
  	D(\cG) = \Big\{f \in \cC^2(\Gamma^0): \,\cG f\in \cC(\Gamma^0),\eta \cG(0)+p_+\partial_+ f(0 )+p_-\partial_- f(0 )=0\Big\}\label{eq:s1_eq3}
 \end{equation}
 where    $p_\pm=\mu_\pm\gamma_{0,\pm}$ and $\partial_\pm$ denote the outward directional derivatives in $x=0$. The coefficient  $p_\pm$ are the probability that the process, being at $x=0$, enters respectively in $\Gamma^0_\pm$. By making use of the It\^o-McKean technique, the authors are able to prove the following.

\begin{thm}[{\cite[Theorems 3.2 and 3.3]{SS2017}}]
 	Assume that $\sigma_\pm >0$ and $b_\pm$ is uniformly Lipschitz continuous in $\Gamma_\pm$.
 	\begin{itemize}
 		\item[(i)] Let $X(t)$ be a diffusion associated with the infinitesimal operator $\cG$ defined in \eqref{eq:s1_eq1},\eqref{eq:s1_eq3}, started at $x \in \Gamma^0$. Then there exists a Brownian motion $W(t)$ such that the process $X$ solves the SDE
 		\begin{equation}\label{eq:s1_eq4}
          \begin{cases}
          dX(t) &= (b(X(t))dt + \sigma dW(t))\mathbbm{1}_{\{X(t)\not =0\}} + (p_+-p_-)L^X(t,0), \\
          X(0) &= x,
        \end{cases}
     \end{equation}
        where $b(x)=b_{\pm}(x)$  and  $\sigma= \sqrt{2\mu_\pm}$ for $x\in\Gamma^0_\pm$ and $L^X(\cdot,0)$ is the local-time of $X$ at $0$.
 		\item[(ii)] The pair $(X,L^X(\cdot,0))$ solving the SDE \eqref{eq:s1_eq4}  is unique in law under the additional constraint
 		\begin{equation}\label{eq:s1_eq5}
          \eta L^X(t,0)=     \int_0^t \mathbbm{1}_{\{X(s) = 0\}}ds.
        \end{equation}
 	\end{itemize}
 \end{thm}
  Observe that solutions to   \eqref{eq:s1_eq4} are not unique.  For example, the SDE is satisfied by an undelayed process ($\eta = 0$) since this process spends almost no time at $0$. Uniqueness, in weak sense, is recovered adding to the process the corresponding local time at $0$. Notice also that \eqref{eq:s1_eq5} implies   that the occupation time of the sticky process at $0$ is positive and we recover that it is null in the non sticky case by formally taking the limit $\eta \to 0$.

\iffalse
%%%
 \begin{rem}
 	For a undelayed diffusions, i.e. $\eta=0$, the amount of time that the process $X$ spend at $x=0$ has zero measure, i.e.
 	$    \int_0^t \mathbbm{1}_{\{X(s) = 0\}} ds = 0$  for all $t>0$  with probability 1. A sticky diffusion process can be always rewritten as
 	\[X(t)=Y(r^{-1}(t))\]
 	where $Y(t)$ is a standard diffusion process  and $r^{-1}$ is the functional inverse of the strictly increasing function $r(t)=t+\eta L^X(t,0)$. Since $r^{-1}(t)<t$ for every $t>0$, the process $Y(t)$ is slowed down when it is at $x=0$. Because of this property,	
 	  $   \int_0^t \mathbbm{1}_{\{X(s) = 0\}} ds$  has positive probability of being strictly positive for all $t>0$ and, yet,  the process does not stay at zero for any positive  interval of time.  		   Note that, away from the vertex, $X$ behaves as a standard diffusion process.  
 \end{rem}
 %%%
 \fi
% In the following lemma, we give an It\^o formula for $X(t)$ (see \cite[Lemma 3.5]{SS2017}).
% \begin{lem} \label{lem:explicit-martingale-prob}
% 	If $X(t)$   is the solution of \eqref{eq:s1_eq4}-\eqref{eq:s1_eq5} and $f \in \cC^2(\RR) $ with  $\cG f\in \cC(\RR)$, then
% 	\begin{equation} \label{eq:s1_eq6}
% 		\begin{split}
% 			f(X(t)) - f(x) &= \int_0^t \cG f(X(s))ds + \int_0^t \sigma \mathbbm{1}_{\{X(s) \not = 0\}}dW(s)   \\
% 			&+(p_+\partial_+f(0)+p_-\partial_-f(0) + \eta \cG f(0))L^X(t,0).
% 		\end{split}
% 	\end{equation}
% \end{lem}

The extension of the previous results to the case of a star graph is considered in \cite{BC2024}. Since the behavior of the process is purely local, the results from \cite{BC2024} can be generalized to general networks and we summarize in the following theorem the properties that are relevant to this work.

 \begin{thm}\label{thm:ito_formula}
  Let $X$ be the Feller process on the network $\Gamma$ generated by \eqref{eq:s1_eq1}, \eqref{eq:s1_eq1a}. Then
  \begin{enumerate}[label=\rm{(\roman*)}]
   \item  there exists a one-dimensional Brownian motion $W$ such that, for every $T>0$, $f \in \cC([0,T] \times \Gamma)$ with $\partial_t f \in \cC([0,T] \times \Gamma)$ and $\partial_x f, \, \partial_x^2 f \in \cC([0,T],PC(\Gamma))$, and  $0 < t \leq T$, we have
   \begin{align*}
  f(t,X(t)) & = f(0,X(0)) + \int_0^t \left ( \partial_t f(s,X(s)) +  \cG_\alpha f(s,X(s)) \right ) \mathbbm{1}_{\{ X(s) \in\Gamma_\alpha\setminus\cV \}}\, ds \\ & \quad + 2 \int_0^t \sqrt \mu_\alpha \partial_x f(s,X(s)) \mathbbm{1}_{\{ X(s) \in\Gamma_\alpha\setminus\cV \}} d W_s \\
   & \quad + \sum_{\varv \in \cV} \int_0^t \left ( \eta_\varv \partial_t f(s,\varv) - \sum_{\alpha \in \cA_\varv} \mu_\alpha \gamma_{\varv,\alpha} \partial_\alpha f(s,\varv) \right ) d L^X(s,\varv),
 \end{align*}
 where, for each $\varv \in \cV$, $L^X(\cdot,\varv)$ is a finite variation process and we use the convention that $\eta_\varv = 0$ if $\varv \in   \partial \cV$;

\item for every bounded measurable function $g \colon \RR_+ \to \RR$ we have
\begin{equation*}
 \int_0^t g(s) \mathbbm{1}_{\{X(s) = \varv\}} \, ds = \eta_{\varv} \int_0^t g(s) dL^X(s,\varv).
\end{equation*}

\end{enumerate}

 \end{thm}

%%%%%%%%%%%%%%%%%%%%%%%%%%%%%
We come back to the study of the Markov process associated to \eqref{eq:s1_eq1}-\eqref{eq:s1_eq1a} and we are   interested in characterizing the corresponding stationary distribution, i.e. a probability measure $ \mf \in \cP_1(\Gamma)$ such that (see \cite[Chapter 4]{EK1986} for instance)
\begin{equation*}
	\int \cG f(x)  \mf(dx) = 0 \quad \textnormal{for all } f \in D(\cG).
\end{equation*}
This characterization was previously obtained in \cite{ACG2021}. We include its proof following a slightly different approach for the sake of completeness. We will show that the measure $\mf$ splits into an absolutely continuous part and a sum of Dirac masses concentrated at the vertices, i.e.
\begin{equation}\label{eq:mu}
	\mf = m \leb + \sum_{\varv \in \cV \setminus \partial \cV} \eta_\varv T_\varv[m] \delta_{\varv},
\end{equation}
where the density $m \in \bssob$
%, for  $\bgamma = \left \{ \gamma_{\varv,\alpha} : \varv \in \cV, \alpha \in \cA_\varv\right \}$, 
is a weak solution (see \cref{defi:weak_FP} below) to
\begin{equation}\label{eq:FP}
	\begin{cases}
		- \mu_\alpha \partial^2 m(x) - \partial \left( b(x) m(x) \right ) = 0 \quad & \textnormal{for all } x \in \Gamma_\alpha,\, \alpha \in \cA, \\[4pt]
		\frac{m_{|\Gamma_\alpha}(\varv)}{\gamma_{\varv,\alpha}} = \frac{m_{|\Gamma_\beta}(\varv)}{\gamma_{\varv,\beta}} =: T_\varv[m] \quad & \textnormal{for all } \alpha, \beta \in \cA_\varv, \, \varv \in \cV \setminus \partial \cV, \\[4pt]
		\sum_{\alpha \in \cA_\varv} \mu_\alpha \partial_\alpha m_{|\Gamma_\alpha}(\varv) + n_{\varv,\alpha}  m_{|\Gamma_\alpha}(\varv)  b_{|\Gamma_\alpha}(\varv) = 0 \quad &  \textnormal{for all } \varv \in \cV, \\[4pt]
		m \geq 0, \, 1 \geq \int_\Gamma m\, dx = 1 -  \sum_{\varv \in \cV \setminus{\partial \cV}} \eta_\varv T_\varv[m] \geq 0.
	\end{cases}
\end{equation}
Note that the second line of the previous system define, up to the factors $\eta_\varv$, the coefficients of the Dirac masses at the vertices.
\begin{defi}\label{defi:weak_FP}
A weak solution to \eqref{eq:FP}  is a function $m\in \bsob{1}{2}{\bgamma}$ such that
\begin{equation}\label{eq:FP_weak}
	\int_\Gamma \mu \partial m \partial v + b m \partial v \ dx = 0 \quad \textnormal{for every } v \in \cssob
\end{equation}	
and $m \geq 0$, $ 1 \geq \int_\Gamma m\, dx = 1 -  \sum_{\varv \in \cV \setminus{\partial \cV}} \eta_\varv T_\varv[m] \geq 0$.
\end{defi}

We refer the reader to \cite{ADLT2019} for a justification of the weak formulation of \eqref{eq:FP}.   Recall the notation $\bgamma := \left \{ \gamma_{\varv, \alpha} : \alpha \in \cA_{\varv}, \varv \in \cV \setminus \partial \cV  \right \}$.
\begin{lem}\label{lem:FP_estimate}
	Assume that $b \in \Lx{\infty}$ and let $m \in \bssob$ satisfy \eqref{eq:FP_weak}. Then there exists a positive constant $C = C(\mu,\norm{b}{L^\infty}, \bgamma, \Gamma)$ such that
	\[
	\norm{m}{\bssob} \leq C \norm{m}{L^1}.
	\]
\end{lem}
\begin{proof}
	We consider the function $\psi \in  PC (\Gamma)$ uniquely determined by $\psi_{|\Gamma_\alpha}(\varv) = \gamma_{\varv, \alpha}$ for every $\alpha \in \cA_\varv$ and $\varv \in \cV \setminus \partial \cV$, $\psi_{\alpha}$ is affine if $\Gamma_\alpha \cap \partial \cV = \varnothing$ and constant otherwise. Then, the function $m\psi$ belongs to $\cssob$. Using $m\psi$ as a test-function in \eqref{eq:FP_weak}, we obtain, for any $\varepsilon > 0$,
	\begin{equation}\label{eq:FP_estimate_1}
		\begin{split}
			(\min \bgamma) \int_\Gamma \mu \module{\partial m}^2\ dx & \leq \int_\Gamma \mu \module{m \partial m \partial \psi} + \module{bm(m \partial \psi + \partial m \psi)}\, dx \\
			& \leq C  \int_\Gamma \module{m \partial m} + \module{m}^2 \, dx \leq C \left (\varepsilon \norm{\partial m}{L^2}^2 + C_\varepsilon \norm{m}{L^2} \right ).
		\end{split}
	\end{equation}
	As a consequence of the continuous embedding $\sob{1}{1} \hookrightarrow \Lx{\infty}$ we have
	\begin{equation}\label{eq:FP_estimate_2}
		\norm{m}{L^2} \leq \norm{m}{L^\infty}^{\half} \norm{m}{L^1}^{\half} \leq C_1 \norm{m}{W^{1,1}}^{\half} \norm{m}{L^1}^{\half} \leq C_2 \left (\norm{\partial m}{L^2} + \norm{m}{L^1} \right)^{\half}\norm{m}{L^1}^{\half}.
	\end{equation}
	Combining \eqref{eq:FP_estimate_1} and \eqref{eq:FP_estimate_2} and using Young's inequality we deduce
	\[
	(\min \bgamma) \underline{\mu} \norm{\partial m}{L^2}^2 \leq 2 \varepsilon \norm{\partial m}{L^2}^2 + \tilde C_\varepsilon \norm{m}{L^1}^2,
	\]
	where $\underline{\mu} = \min_{\alpha \in \cA} \mu_\alpha > 0$. Choosing $\varepsilon$ small enough we obtain the desired inequality.
\end{proof}

\begin{lem}\label{lem:FP_theta}
	Let $b \in PC(\Gamma) $. Then, for every $\vartheta \geq 0$, there exists unique $m^\vartheta \in \bssob$ such that $m^{\vartheta} \geq 0$, $\int_\Gamma m^\vartheta\, dx = \vartheta$ and satisfying \eqref{eq:FP_weak}. Furthermore, $m^0 = 0$ and $m^\vartheta > 0$ for $\vartheta > 0$.
\end{lem}
\begin{proof}
	From \cref{lem:FP_estimate} it is clear that $m^0 = 0$ is the unique element in $\bssob$ satisfying the conditions in the statement. From \cite[Theorem 2.7]{ADLT2019} we know that there exists a unique $m^1$ satisfying \eqref{eq:FP_weak} with $\int_\Gamma m^1\, dx = 1$. Moreover we know that $m^1 >0$, in particular $T_\varv[m] > 0$ for every $\varv \in \cV \setminus \partial \cV$. Then, setting $m^\vartheta = \vartheta m^1$, we obtain an element in $\bssob$ satisfying \eqref{eq:FP_weak}, $m^\vartheta > 0$ and $\int_\Gamma m^\vartheta \, dx = \vartheta$.
\end{proof}

\begin{prop}\label{prop:FP_well_posed}
	Let $b \in PC(\Gamma) $. Then there exists a unique weak solution $m \in \bssob$ to \eqref{eq:FP}.
\end{prop}
\begin{proof}
	For existence, we consider the mapping $\Phi \colon [0,1] \to \RR_+$ defined by
	\[
	\Phi(\vartheta) = \vartheta + \sum_{\varv \in \cV \setminus \partial \cV} \eta_\varv T_\varv[m^\vartheta],
	\]
	where $m^\vartheta$ is given in \cref{lem:FP_theta}. We claim that the mapping $\Phi$ is continuous. Indeed, let $(\vartheta_n)_{n \geq 0}$ be a sequence in $[0,1]$ converging to some $\vartheta \in [0,1]$ as $n$ tends to infinity. Using \cref{lem:FP_theta}, for each $n \geq 0$, there exists $m^n \in \bssob$ such that $\int m^n\, dx = \vartheta_n$, $m^n \geq 0$ and satisfies \eqref{eq:FP_weak}. Using \cref{lem:FP_estimate}, we see that the sequence $(m^n)_{n\geq 0}$ is bounded in $\bssob$ and we may therefore extract a subsequence, which we still denote by $m^n$, converging weakly in $\bssob$ and strongly in $PC(\Gamma) $ to some $m \in \bssob$ as $n$ tends to infinity. It follows that $m$ also satisfies \eqref{eq:FP_weak}, is non-negative and $\int m\, dx = \vartheta$. From uniqueness in \cref{lem:FP_theta} we conclude to $m = m^\vartheta$ and that the whole sequence converges to $m^\vartheta$. This proves the continuity of the mapping $\vartheta \mapsto m^\vartheta$ from $[0,1]$ to $PC(\Gamma) $. The continuity of $\Phi$ then follows from the continuity of $PC(\Gamma) \ni \tilde m \mapsto T_\varv[\tilde m] \in \RR$ for every $\varv \in \cV\setminus \partial \cV$.

	Notice that $\Phi(0) = 0$ and that $\Phi(\vartheta) > \vartheta$ for $\vartheta > 0$ since each $m^\vartheta$ is strictly positive. From the intermediate value theorem, we conclude that there exists $\bar \vartheta \in [0,1]$ such that
	\[
	\Phi(\bar \vartheta) = \bar \vartheta + \sum_{\varv \in \cV \setminus \partial \cV} \eta_\varv T_\varv \left [m^{\bar \vartheta} \right ] = \int_\Gamma m^{\bar \vartheta}\, dx +  \sum_{\varv \in \cV \setminus \partial \cV} \eta_\varv T_\varv \left [m^{\bar \vartheta} \right] = 1.
	\]
	This proves existence.

	We now prove uniqueness. Notice first that we cannot have $\sum_{\varv \in \cV \setminus \partial \cV} \eta_\varv T_\varv \left [m \right] = 1$ since, otherwise, \cref{lem:FP_theta} implies $m=0$, which is a contradiction. Let $m_1$ and $m_2$ be two weak solutions to \eqref{eq:FP}. Up to relabelling, we may assume that
	\begin{equation}\label{eq:FP_uniqueness}
		1 > \sum_{\varv \in \cV \setminus \partial \cV} \eta_\varv T_\varv \left [m_1 \right] \geq \sum_{\varv \in \cV \setminus \partial \cV} \eta_\varv T_\varv \left [m_2 \right].
	\end{equation}
	Set $w_i = \left (1 - \sum_{\varv \in \cV \setminus \partial \cV} \eta_\varv T_\varv \left [m_i \right] \right)^{-1} m_i$ for $i=1,2$. Then $w_i$ satisfies \eqref{eq:FP}, $w_i \geq 0$ and $\int w_i\, dx = 1$. From uniqueness in \cite[Theorem 2.7]{ADLT2019} we deduce that $w_1 =  w_2=: \bar m$. In particular
	\begin{equation}\label{eq:FP_uniqueness_2}
	m_i = \left (1 - \sum_{\varv \in \cV \setminus \partial \cV} \eta_\varv T_\varv \left [m_i \right] \right)\bar m.
	\end{equation}
	From \eqref{eq:FP_uniqueness} we obtain $m_1 \leq m_2$ and therefore
	\[
	\sum_{\varv \in \cV \setminus \partial \cV} \eta_\varv T_\varv \left [m_1 \right] = \sum_{\varv \in \cV \setminus \partial \cV} \eta_\varv T_\varv \left [m_2 \right].
	\]
	 Uniqueness then follows from \eqref{eq:FP_uniqueness_2}.
\end{proof}

\begin{thm}
	The measure $\mf$ defined by \eqref{eq:mu} is a stationary distribution for the Markov process generated by \eqref{eq:s1_eq1} and \eqref{eq:s1_eq1a}.
\end{thm}
\begin{proof}
	Let $f \in D(\cG)$. Integrating by part and using the fact that $m \in \bssob$, we compute
	\begin{align*}
	 - \int_{\Gamma} \mu \partial^2 f(x) m(x) \, dx &= - \sum_{\alpha \in \cA} \int_{0}^{L_\alpha} \mu_\alpha \partial^2 f_\alpha(x) m_\alpha(x)\, dx \\
	 & = - \sum_{\alpha \in \cA} \mu_\alpha \left [ \partial f_\alpha(x) m_\alpha(x) \right]_{x=0}^{x=L_\alpha} + \sum_{\alpha \in \cA} \int_0^{L_\alpha} \mu_\alpha \partial f_\alpha(x) \partial m_\alpha(x)\, dx \\
	 & = - \left [\sum_{\varv \in \cV \setminus \partial \cV} \sum_{\alpha \in \cA_\varv} \mu_\alpha  \gamma_{\varv_\alpha} \partial_\alpha f(\varv) T_\varv[m] \right ] + \int_\Gamma \mu \partial f(x) \partial m(x)\, dx.
	\end{align*}
	Since $f \in D(\cG)$ and $m$ satifies \eqref{eq:FP_weak}, it follows that
	\begin{equation*}
		\begin{split}
			- \int_\Gamma \cG f(x) \mf(dx) & = - \int_\Gamma \cG f(x)m(x)\ dx - \sum_{\varv \in \cV\setminus \partial \cV} \eta_\varv  T_\varv[m] \cG f(\varv) \\  = \int_\Gamma (\mu \partial f \partial m& + b \partial f m )\, dx - \sum_{\varv \in \cV \setminus \partial \cV} T_\varv[m] \left ( \eta_\varv \cG f(\varv) + \sum_{\alpha \in \cA_\varv} \mu_\alpha \gamma_{\varv,\alpha} \partial_\alpha f(\varv) \right ) =0.
		\end{split}
	\end{equation*}
\end{proof}

\begin{rem}
 Observe that in the undelayed case, i.e. $\eta_\varv = 0$ for every $\varv \in \cV \setminus \partial \cV$, then we recover the stationary distribution obtained in \cite{ADLT2019}.
\end{rem}

\begin{prop}[Stability of $m$]\label{prop:FP_stability}
  Given $b^n \in PC(\Gamma)$, $n \in \NN$, let $m^n \in \bssob$ be the corresponding solution to \eqref{eq:FP}, given by \cref{prop:FP_well_posed}. Assume that there exists $b \in PC(\Gamma)$ such that $b^n_\alpha$ converges uniformly to $b_\alpha$ for every $\alpha \in \cA$. Then there exists $m \in \bssob$ such that 
  \begin{equation}\label{eq:stability}
   \begin{cases}
    m^n_\alpha \to m_\alpha \quad & \textnormal{uniformly, for every } \alpha \in \cA, \\
    m^n \rightharpoonup m \quad & \textnormal{weakly in } \bssob.
   \end{cases}
  \end{equation}

\end{prop}

\begin{proof}
  We first consider arbitrary subsequences, still denoted by $b^n$ and $m^n$. We know that $m^n \geq 0$ and $0 \leq \int_\Gamma m^n \, dx \leq 1$ for every $n \in \NN$. In particular $(m^n)_{n \in \NN}$ is bounded in $\Lx{1}$. It then follows from \cref{lem:FP_estimate} that it is also bounded in $\bssob$. We deduce that there exists a strictly increasing map $\iota \colon \NN \to \NN$ and $m \in \bssob$ such that $m^{\iota(n)} \rightharpoonup m$ weakly in $\bssob$ and $m^{\iota(n)}_\alpha \to m_\alpha$ uniformly.
  We can pass to the limit in the weak formulation \eqref{eq:FP_weak} to deduce that $m$ also satisfies \eqref{eq:FP_weak}. In addition we also have
  \[
   1 =  \lim_{n \to \infty} \left [ \int_\Gamma m^{\iota(n)} \, dx  + \sum_{\varv \in \cV \setminus{\partial \cV}} \eta_\varv T_\varv[m^{\iota(n)}] \right ] = \int_\Gamma m \, dx  + \sum_{\varv \in \cV \setminus{\partial \cV}} \eta_\varv T_\varv[m].
  \]
  This proves that $m$ is a solution \eqref{eq:FP}. Since this solution is unique we have proven that every subsequence of $(m^n)_{n \in \NN}$ has a further subsequence converging to an unique limit. We conclude that the full sequence converges to $m$ in the topologies mentioned in \eqref{eq:stability}.

\end{proof}

\begin{prop}[Stability of $\mf$]\label{prop:continuity_density}
 Let $(m^n)_{n \in \NN}$ be a sequence in $PC(\Gamma)$ satisfying
 \[
  m^n \geq 0 \quad \textnormal{and} \quad \int_\Gamma m^{n} \, dx  + \sum_{\varv \in \cV \setminus{\partial \cV}} \eta_\varv T_\varv[m^n] = 1 \quad \textnormal{for all } n \in \NN,
 \]
 and converging uniformly to some $m \in PC(\Gamma)$. Define $(\mf^n)_{n \in \NN}$ and $\mf$ according to \eqref{eq:mu}. Then $\mf_n$ converges to $\mf$ in $\cP_1(\Gamma)$.
\end{prop}
\begin{proof}
Notice first that the uniform convergence of $m^n_\alpha$ to $m_\alpha$ implies that $m \geq 0$ and that
  \[
   1 =  \lim_{n \to \infty} \left [ \int_\Gamma m^n \, dx  + \sum_{\varv \in \cV \setminus{\partial \cV}} \eta_\varv T_\varv[m^n] \right ] = \int_\Gamma m \, dx  + \sum_{\varv \in \cV \setminus{\partial \cV}} \eta_\varv T_\varv[m].
  \]
It follows that we indeed have $\mf \in \cP_1(\Gamma)$. Let $\psi \in \Lip(\Gamma)$,   the space of functions from $\Gamma$ to $\RR$ which are Lipschitz continuous with respect to the geodetic metric on $\Gamma$. We compute
  \begin{align*}
   & \int_\Gamma \psi(x)(\mf - \mf^n)(dx) = \int_\Gamma \psi(x)(m(x) - m^n(x)) dx + \sum_{\varv \in \cV \setminus{\partial \cV}} \eta_\varv \psi(\varv) \left (T_\varv[m] - T_\varv[m^n] \right ) \\
   & \qquad \leq \norm{\psi}{\Lip(\Gamma)} \diam(\Gamma) \left( \norm{m - m^n}{L^1(\Gamma)} + \sum_{\varv \in \cV \setminus{\partial \cV}} \eta_\varv \module{T_\varv[m] - T_\varv[m^n]} \right ).
  \end{align*}
  From standard properties of the Kantorovich-Rubinstein metric (see \cite[Remark 6.5]{V2009} for instance) we conclude that $\lim_{n \to \infty} {\bf d_1}(\mf^n,\mf) = 0$.
\end{proof}

%%%%%%%%%%%%%%%
%             %
%%%%%%%%%%%%%%%

\section{The Hamilton-Jacobi-Bellman equation}\label{sec:4}
This section is devoted to the study of the Hamilton-Jacobi-Bellman equation entering in the Mean Field Game system. Moreover,   we provide a justification of the transition condition at the vertices by proving a verification theorem for a discounted infinite horizon optimal control problem, whose dynamics is given by a controlled sticky diffusion process on the network.

We consider  the discounted Hamilton-Jacobi-Bellman equation
\begin{equation}\label{eq:HJB_disc}
	\begin{cases}
		- \mu_\alpha  \partial^2 u(x) + H(x,\partial u(x)) +\lambda u(x) = F(x) \quad & \textnormal{for all } x \in \Gamma_\alpha \setminus \cV,\, \alpha \in \cA, \\
		u_{|\Gamma_\alpha}(\varv) = u_{|\Gamma_\beta}(\varv) \quad & \textnormal{for all } \alpha, \beta \in \cA_\varv, \, \varv \in  \cV, \\
		  \sum_{\alpha \in \cA_\varv}  \mu_\alpha \gamma_{\varv, \alpha} \partial_\alpha u(\varv) =\eta_\varv (  \theta_\varv-\lambda u(\varv) ) \quad & \textnormal{for all } \varv \in \cV \setminus \partial \cV, \\
		\partial_\alpha u(\varv) = 0 \quad & \textnormal{for all } \varv \in \partial \cV,
	\end{cases}
\end{equation}
and   the corresponding ergodic problem
\begin{equation}\label{eq:HJB}
\begin{cases}
 - \mu_\alpha  \partial^2 u(x) + H(x,\partial u(x)) + \rho = F(x) \quad & \textnormal{for all } x \in \Gamma_\alpha \setminus \cV ,\, \alpha \in \cA, \\
 u_{|\Gamma_\alpha}(\varv) = u_{|\Gamma_\beta}(\varv) \quad & \textnormal{for all } \alpha, \beta \in \cA_\varv, \, \varv \in  \cV, \\
   \sum_{\alpha \in \cA_\varv}   \mu_\alpha \gamma_{\varv, \alpha} \partial_\alpha u(\varv) = \eta_\varv( \theta_\varv-\rho) \quad & \textnormal{for all } \varv \in \cV \setminus \partial \cV, \\
 \partial_\alpha u(\varv) = 0 \quad & \textnormal{for all } \varv \in \partial \cV.
\end{cases}
\end{equation}
 The second line in \eqref{eq:HJB_disc} and \eqref{eq:HJB} gives the continuity of $u$ at the vertices, while the following two lines are    a generalized Kirchhoff condition at the internal vertices and a Neumann condition at the boundary vertices.\\
We will make the following assumptions.
\begin{enumerate}[label={\bf (H\arabic*)}]
	\item \label{h:HJB1}
	 We assume that $\lambda > 0$ and that, for each   $\varv \in \cV \setminus \partial \cV$ and $\alpha \in \cA_\varv$,  $\mu_\alpha>0$, $\gamma_{\varv,\alpha}>0$, $\eta_\varv\ge 0$  and $\theta_\varv\in\RR$. Moreover
	\[
	\sum_{\alpha \in \cA_\varv} \mu_{\alpha} \gamma_{\varv,\alpha} = 1 \quad \textnormal{for all } \varv \in \cV .
	\]
\item  For each $\alpha \in \cA$,
 \begin{itemize}
  \item \label{h:HJB2} The mapping $H_{\alpha} \colon \Gamma_\alpha \times \RR \to \RR$, defined by
 \begin{equation*}
  H_\alpha(x,p) = \begin{cases}
                   H(x,p) \quad & \textnormal{if } x \in \Gamma_\alpha \setminus \cV, \\
                   \lim\limits_{\substack{y \to x \\ y \in \Gamma_\alpha \setminus \cV}} H(y,p) \quad & \textnormal{if } x \in \cV,
                  \end{cases}
 \end{equation*}
 is well-defined and continuous; moreover, for each $x \in \Gamma_\alpha \setminus \cV$, the mapping $p \mapsto H_\alpha(x,p)$ is continuously differentiable and $\partial_p H_\alpha$ can be extended to a continuous mapping on $\Gamma_\alpha \times \RR$.

 \item   There exists a constant $C_H > 0$ and $q \in (1,2]$ such that
 \begin{align}
  \module{H_\alpha(x,p)} \leq C_H \left ( 1 + \module{p}^q \right) \quad & \textnormal{for all } (x,p) \in \Gamma_\alpha \times \RR, \label{eq:H_growth} \\
  \module{\partial_p H_{\alpha}(x,p)} \leq C_H \left (1 + \module{p}^{q-1} \right) \quad & \textnormal{for all } (x,p) \in \Gamma_\alpha \times \RR, \label{eq:H_p_growth} \\
  H_\alpha(x,p) \geq C_H^{-1} \module{p}^q - C_H \quad & \textnormal{for all } (x,p) \in \Gamma_\alpha \times \RR. \label{eq:H_coercive}
 \end{align}
 \end{itemize}
\item \label{h:HJB3} $F$ belongs to $PC^{\varsigma} (\Gamma)$ for some $\varsigma \in (0,1)$.
 \end{enumerate}
% In what follows it will be convenient to introduce the following notation
% \[
%  \bgamma = \left \{ \gamma_{\varv,\alpha} : \, \varv \in \cV \setminus \partial \cV, \, \alpha \in \cA_\varv \right \}.
% \]

%%%%%
\begin{prop}
Assume \ref{h:HJB1}-\ref{h:HJB3}. Then, there exists a  solution $u \in \cC^{2,\varsigma}(\Gamma)$ to \eqref{eq:HJB_disc}.
\end{prop}
\begin{proof}
 The following argument was introduced for star-shaped networks in \cite{O2021} and extended  to general networks with Kirchhoff conditions in \cite{BLT2024}. Let us write $\mathbb{V} = \cV \setminus \partial \cV$. Let $\Phi \colon \RR^{\mathbb{V}} \to\cC^2(\Gamma)$ with $\Phi(z) = u^z$, where
 \begin{equation}\label{eq:Ohavi-trick}
  \begin{cases}
   -\mu_\alpha \partial^2 u_\alpha^z(x) + H_\alpha(x, \partial u_\alpha^z(x)) + \lambda u_\alpha^z(x) =  F_\alpha(x) \quad \textnormal{for all } x\in \Gamma_\alpha\setminus\cV, \\
   u_\alpha^z(0) = z_\varv \textnormal{ if } \varv = \pi_\alpha(0) \in \mathbb{V} \textnormal{ and } \partial u_\alpha^z(0) = 0 \textnormal{ if } \varv \in \partial \cV, \\
   u_\alpha^z(L_\alpha) = z_\varv \textnormal{ if } \varv = \pi_\alpha(L_\alpha) \in \mathbb{V} \textnormal{ and } \partial u_\alpha^z(L_\alpha) = 0 \textnormal{ if } \varv \in \partial \cV.
  \end{cases}
 \end{equation}
 System  \eqref{eq:Ohavi-trick} is a family of  PDEs defined for each edge $\Gamma_\alpha$ coupled by means of the boundary value  $u^z_\alpha(\varv)=z_\varv$ for all $\alpha \in \cA_\varv$, $\varv\in \mathbb{V}$.
  From the standard theory of quasi-linear elliptic equations we know that this mapping is well-defined and that
 \[
  \max_{\alpha \in \cA} \norm{u_\alpha^z}{\cC^{2,\varsigma}([0,L_\alpha])} \leq C \left ( \module{z}_{\infty} + \max_{\alpha \in \cA} \norm{F_\alpha}{\cC^{0,\varsigma}([0,L_\alpha])} \right ).
 \]
 A straightforward adaptation of \cite[Proposition B.1]{O2021} shows that $\Phi$ is continuous. We then set
 \[
  K := \max \left \{ \norm{F}{L^\infty} + \norm{H(\cdot, 0)}{L^\infty}, \max_{\varv \in \cV}  \module{\theta_\varv} \right \}.
 \]
 For every $z \in \RR^\mathbb{V}$, one can see that the constant function
 \[
  \phi_\alpha^z := \max \left \{ K/\lambda,\, \max_{\varv \in \cV} \module{z} \right \}
 \]
 is a super-solution to \eqref{eq:Ohavi-trick} for every $\alpha \in \cA$. From a standard comparison  principle in each $\Gamma_\alpha$ we deduce that $u_\alpha^z \leq \phi_\alpha^z$ for every $\alpha \in \cA$. Fix $\varv_0 \in \cV \setminus \partial \cV$, $M = \max \{ K/\lambda, K \}$ and $z^0$ be such that
 \[
  \begin{cases}
   z^0_\varv = M \quad & \textnormal{if } \varv = \varv_0, \\
   z^0_\varv \leq  M \quad & \textnormal{otherwise.}
  \end{cases}
 \]
 In this case we have $\phi_\alpha^{z^0} \equiv M$ for every $\alpha \in \cA$. Since $u_\alpha^{z_0} \leq \phi_\alpha^{z_0}$ and $u_\alpha^{z_0}(\pi_\alpha^{-1}(\varv_0)) = M = \phi_\alpha^{z_0}(\pi_\alpha^{-1}(\varv_0))$,  it follows that $\pi_\alpha^{-1}(\varv_0)$ is a maximum point of $u^{z^0}_\alpha$ for every $\alpha \in \cA_{\varv_0}$, and therefore $  \partial u_\alpha^{z^0}(\varv_0) \geq 0$. Hence, recalling the definition of $K$,    we have
 \[
  \eta_\varv \lambda z^0_\varv + \sum_{\alpha \in \cA_\varv}   \mu_\alpha \gamma_{\varv, \alpha} \partial u_\alpha^{z^0}(\pi_\alpha^{-1}(\varv)) \geq \eta_\varv \theta_\varv.
 \]
 An analogous conclusion can be repeated for every $\varv \in \mathbb V$. Similarly, if $ z^0$ is such that
 \[
  \begin{cases}
   z^0_\varv = -M \quad & \textnormal{if } \varv = \varv_0, \\
   z^0_\varv \geq - M \quad & \textnormal{otherwise.}
  \end{cases}
 \]
 we obtain, by replacing $\phi_\alpha^z$ with $- \phi_\alpha^z$ which is a sub-solution to \eqref{eq:Ohavi-trick} for every $\alpha \in \cA$, that
 \[
  \eta_\varv \lambda z^0_{\varv_0} + \sum_{\alpha \in \cA_\varv}  \mu_\alpha \gamma_{\varv, \alpha} \partial u_\alpha^{z^0}(\pi_\alpha^{-1}(\varv_0)) \leq \eta_\varv \theta_\varv.
 \]
 For each $\varv \in \mathbb V$ we consider the continuous mapping $\Psi_\varv \colon \cC^1(\Gamma) \to \RR$ defined by
 \[
  \Psi_\varv (u) = \eta_\varv \lambda u (\varv) + \sum_{\alpha \in \cA_\varv} \mu_\alpha \gamma_{\varv, \alpha} \partial_\alpha u(\varv).
 \]

 We have proven that, for each $\varv \in \mathbb{V}$, the mapping $\Psi_\varv \circ \Phi \colon [-M,M]^{\mathbb{V}} \to \RR$ is continuous and satisfies
 \[
  \begin{cases}
   \Psi_\varv \circ \Phi(z) \geq \eta_\varv \theta_\varv \quad & \textnormal{if } z_\varv = M, \\
   \Psi_\varv \circ \Phi(z) \leq \eta_\varv \theta_\varv \quad & \textnormal{if } z_\varv = - M.
  \end{cases}
 \]
 We can apply the Poincaré-Miranda theorem \cite{M1940,M2013} to the mapping $[-M,M]^{\mathbb{V}} \ni z \mapsto \left (\Psi_\varv \circ \Phi(z) \right)_{\varv \in \mathbb V} \in \RR^{\mathbb V}$ to conclude that there exists $z^\star \in [-M,M]^{\mathbb V}$ such that $\Psi_\varv \circ \Phi(z^\star) = \eta_\varv \theta_\varv$ for every $\varv \in \mathbb V$. This concludes the proof.
\end{proof}
%%%% 
To show the uniqueness of the solution to \eqref{eq:HJB_disc}, we prove a comparison principle.
\begin{defi}[Sub- and Super-solutions]
\hfill
  \begin{enumerate}
   \item We say that $u \in \cC^2(\Gamma)$ is a sub-solution to \eqref{eq:HJB_disc} if
  \begin{equation*}
  \begin{cases}
  - \mu_\alpha  \partial^2 u(x) + H(x,\partial u(x)) +\lambda u(x) \leq F(x) \quad & \textnormal{for all } x \in \Gamma_\alpha \setminus \cV,\, \alpha \in \cA, \\
  \sum_{\alpha \in \cA_\varv} \mu_\alpha \gamma_{\varv, \alpha} \partial_\alpha u(\varv) \leq \eta_{\varv} \left (\theta_\varv - \lambda u(\varv) \right ) \quad & \textnormal{for all } \varv \in \cV \setminus \partial \cV, \\
  \partial_\alpha u(\varv) \leq 0 \quad & \textnormal{for all } \varv \in \partial \cV.
  \end{cases}
\end{equation*}
  \item We say that $v \in \cC^2(\Gamma)$ is a super-solution to \eqref{eq:HJB_disc} if
  \begin{equation*}
  \begin{cases}
  - \mu_\alpha  \partial^2 v(x) + H(x,\partial v(x)) +\lambda v(x) \geq F(x) \quad & \textnormal{for all } x \in \Gamma_\alpha \setminus \cV,\, \alpha \in \cA, \\
  \sum_{\alpha \in \cA_\varv} \mu_\alpha \gamma_{\varv, \alpha} \partial_\alpha v(\varv) \geq \eta_\varv \left (\theta_\varv - \lambda v(\varv) \right )\quad & \textnormal{for all } \varv \in \cV \setminus \partial \cV, \\
  \partial_\alpha v(\varv) \geq 0 \quad & \textnormal{for all } \varv \in \partial \cV.
  \end{cases}
\end{equation*}
  \end{enumerate}

\end{defi}

\begin{prop}[Comparison principle]\label{prop:comparison}
 Assume \ref{h:HJB1}-\ref{h:HJB2} and let $u$, $v \in \cC^2(\Gamma)$ be sub- and super-solution to \eqref{eq:HJB_disc}, respectively. Then $u \leq v$ in $\Gamma$. In particular there exists at most one classical solution to \eqref{eq:HJB_disc}.
\end{prop}

\begin{proof}
 Set $w = u - v$. Clearly $w \in \cC^2(\Gamma)$  satisfies
 \begin{equation}\label{eq:CP_linearized}
 \begin{cases}
  - \mu_\alpha \partial^2 w_\alpha + b_\alpha \partial w_\alpha + \lambda w_\alpha \leq 0 \quad & \textnormal{in } (0,L_\alpha) \textnormal{ for all } \alpha \in \cA, \\
   \eta_{\varv} \lambda w(\varv)+ \sum_{\alpha \in \cA_\varv} \mu_\alpha \gamma_{\varv, \alpha} \partial_\alpha w(\varv) \leq 0 \quad & \textnormal{for all } \varv \in \cV \setminus \partial \cV,\\
  \partial_\alpha w(\varv) \leq 0 \quad &  \textnormal{for all } \varv \in \partial \cV,
  \end{cases}
 \end{equation}
 where
 \[
  b_\alpha(x) = \int_0^1 \partial_p H_\alpha(x, t u_\alpha(x) + (1 - t) v_\alpha(x)) \, dt \quad \textnormal{for all } x \in (0, L_\alpha).
 \]

 Let $x_0$ be a maximum point of $w$, such a point exists since $\Gamma$ is a compact metric space and $w$ is continuous on $\Gamma$. We may assume that $M := w(x_0) > 0$, since otherwise there is nothing to prove. Moreover, because of the strong maximum principle \cite[Theorem 3.5]{GT2001} and Hopf's lemma \cite[Lemma 3.4]{GT2001} we know that either $w$ is constant in $\Gamma_\alpha$, where $\alpha \in \cA$ is such that $x_0 \in \Gamma_\alpha$ or $x_0 \in \cV$ and $\partial_\alpha w(x_0) > 0$. In the first case, \eqref{eq:CP_linearized} directly yields $M \leq 0$, a contradiction. We must therefore have $x_0 = \varv \in \cV$. Using the same argument for each $\alpha \in \cA_\varv$ we are left with the case $\partial_\alpha w(\varv) > 0$ for every $\alpha \in \cA_{\varv}$. In the case where $\nu  \in \partial \cV$, we also directly obtain a contradiction from \eqref{eq:CP_linearized}. We are henceforth left with the case $ \varv \in \cV \setminus \partial \cV$. We have
 \[
  \sum_{\alpha \in \cA_\varv} \mu_\alpha \gamma_{\varv, \alpha} \partial_\alpha w(\varv) > 0
 \]
 so that \eqref{eq:CP_linearized} implies $\eta_{\varv} \lambda M=\eta_{\varv} \lambda w(\varv) < 0$, also a contradiction. We conclude that $M \leq 0$ and, in particular, that $u \leq v$ in $\Gamma$.
\end{proof}

We now turn to the study of \eqref{eq:HJB}. As  usual, we show existence of a solution passing to the limit for $\lambda\to 0^+$ in \eqref{eq:HJB_disc}.
\begin{lem}\label{cor:sup_bound}
 Assume \ref{h:HJB1}-\ref{h:HJB3} and let $u \in \cC^2(\Gamma)$ be the solution to \eqref{eq:HJB_disc}. There exists $C_1 > 0$, depending only on $\norm{F}{L^\infty}$, $( \theta_\varv)_{\varv \in \cV \setminus \partial \cV}$ and $\norm{H(\cdot,0)}{L^\infty}$, such that
 \begin{equation}\label{stima_Linf}
 	 \norm{\lambda u}{L^\infty} \leq   C_1
 \end{equation}.
\end{lem}
\begin{proof}
	Setting $C=\max_{x\in\Gamma}|H_\alpha(x,0)-F(x)|$ and $C_1=\max\{C, \max\{\theta_\varv\}_{\varv \in \cV \setminus \partial \cV}\}$, we see that the constant functions $-C_1/\lambda$ and $C_1/\lambda$ are respectively a  sub and a supersolution to \eqref{eq:HJB_disc}. Hence the results follows immediately from \cref{prop:comparison}.
\end{proof}
\begin{lem}\label{prop:gradient_bound}
 Assume \ref{h:HJB1}-\ref{h:HJB3} and let $u \in \cC^2(\Gamma)$ the solution to \eqref{eq:HJB_disc}. There exists $C_2 > 0$, depending only on $\norm{F}{L^\infty}$, $(\module{\theta_\varv})_{\varv \in \cV \setminus \partial \cV}$, $\norm{H(\cdot,0)}{L^\infty}$, $C_H$, $\bgamma$, $\Gamma$ and $q$, such that
 \[
  \norm{\partial u}{L^q} \leq  C_2.
 \]
\end{lem}

\begin{proof}
 Let $\psi \colon \Gamma \to \RR_+$ be the unique function   affine in $\Gamma_\alpha$ for each $\alpha \in \cA$, $\psi_{|\Gamma_\alpha}(\varv) = \gamma_{\varv,\alpha}$ for all $\varv \in \cV \setminus \partial \cV$ and  $\alpha \in \cA_\varv$ and  $\psi_\alpha$ constant if the edge $\Gamma_\alpha$ touches the boundary. Multiplying by $\psi$ in the equation satisfied by $u$ and integrating by parts we obtain
 \begin{align*}
  \int_{\Gamma} F \psi \, dx & = \int_\Gamma \left (  - \mu  \partial^2 u(x) + H(x,\partial u(x)) +\lambda u(x) \right) \psi \, dx \\
  & = \int_\Gamma \mu \partial u \partial \psi + H(x,\partial u) \psi + \lambda u \psi\, dx - \sum_{\alpha \in \cA} \left [ \mu_\alpha \partial u_\alpha(x) \psi_\alpha(x) \right]_{x=0}^{x=L_\alpha} \\
  &= \int_\Gamma \mu \partial u \partial \psi + H(x,\partial u) \psi + \lambda u \psi\, dx - \sum_{\varv \in \cV \setminus \partial \cV} n_{\varv,\alpha} \mu_\alpha \gamma_{\varv,\alpha} \partial u_{|\Gamma_\alpha}(\varv) \\
  & = \int_\Gamma \mu \partial u \partial \psi + H(x,\partial u) \psi + \lambda u \psi\, dx - \sum_{\varv \in \cV \setminus \partial \cV} \mu_\alpha \gamma_{\varv,\alpha} \partial_\alpha u(\varv) \\
  & = \int_\Gamma \mu \partial u \partial \psi + H(x,\partial u) \psi + \lambda u \psi\, dx + \big (\sum_{\varv \in \cV \setminus \partial \cV} \lambda u(\varv) - \theta_\varv \big).
 \end{align*}
 Notice that
 \[
  0 < \min \bgamma \leq \psi \leq \max \bgamma < + \infty.
 \]
 Using \ref{h:HJB2} and \eqref{stima_Linf}, we deduce that there exists a positive constant $C > 0$, independent of $\lambda$, such that
 \begin{align*}
  \int_\Gamma \module{\partial u}^q\, dx & \leq C \left( \int_\Gamma \module{\partial u} \, dx + C_1 + \max_{\varv} \module{\theta_\varv} + \norm{F}{L^\infty} \right ) \\
  & \leq  C \left( \norm{\partial u}{L^q}^{1/q} \leb(\Gamma)^{(q-1)/q} + C_1 + \max_{\varv} \module{\theta_\varv} +  \norm{F}{L^\infty} \right ).
 \end{align*}
 The conclusion then follows from Young's inequality.
\end{proof}

\begin{thm}\label{thm:HJB_well_posed}
 Under assumptions \ref{h:HJB1}-\ref{h:HJB3}, there exists a solution $(u,\rho) \in \cC^{2,\varsigma}(\Gamma) \times \RR$ to \eqref{eq:HJB}. Furthermore, this solution is unique up to translation of $u$ by a constant. Finally there exists $\hat C > 0$, depending only on $\norm{F}{L^\infty}$, $(\module{\theta_\varv})_{\varv \in \cV \setminus \partial \cV}$ and $\norm{H(\cdot,0)}{L^\infty}$, $C_H$, $\bgamma$, $\Gamma$ and $q$, such that
 \[
  \norm{\partial u}{\Lx{q}} + \norm{u}{\sob{2}{1}} \leq \hat C.
 \]
\end{thm}
\begin{proof} The argument follows the lines of \cite[Theorem 3.4]{ADLT2019}.

 \underline{Proof of Existence} : For each $k \in \NN$ we set $\lambda_k = 2^{-k}$ and consider $v_k \in \cC^{2,\varsigma}(\Gamma)$ the unique solution to \eqref{eq:HJB_disc} with $\lambda$ replaced by $\lambda_k$. For each $k \in \NN$ there exists $x_k \in \Gamma$ such that $\min_{\Gamma} v_k = v(x_k)$ and we set $u_k = v_k - v_k(x_k)$. Each $u_k$ satisfies
  \begin{equation}\label{eq:HJB_existence_1}
\begin{cases}
 - \mu_\alpha  \partial^2 u_k(x) + H(x,\partial u_k(x)) + \lambda_k u_k(x) = F(x) - \lambda_k v_k(x_k) \quad & \textnormal{for all } x \in \Gamma_\alpha \setminus \cV,\, \alpha \in \cA, \\
 u_{k|\Gamma_\alpha}(\varv) = u_{k|\Gamma_\beta}(\varv) \quad & \textnormal{for all } \alpha, \beta \in \cA_\varv, \, \varv \in  \cV, \\
 \lambda_k \eta_\varv u_k(\varv) + \sum_{\alpha \in \cA_\varv} \mu_\alpha \gamma_{\varv, \alpha} \partial_\alpha u_k(\varv) = \eta_\varv \left ( \theta_\varv - \lambda_k v_k(x_k) \right ) \quad & \textnormal{for all } \varv \in \cV \setminus \partial \cV, \\
 \partial_\alpha u_k(\varv) = 0 \quad & \textnormal{for all } \varv \in \partial \cV.
\end{cases}
\end{equation}
In particular, using \ref{h:HJB2} and \cref{cor:sup_bound}, we have
\begin{align*}
 \module{\partial^2 u_k} & \leq \underline{\mu}{^{-1}} \left ( \module{H(x, \partial u_k)} + \module{F} + \lambda_k \left ( \module{u_k} + \module{v_k(x_k)} \right) \right ) \\
 & \leq \underline{\mu}{^{-1}} \left (C_H \left ( 1 + \module{\partial v_k}^q \right ) + \norm{F}{L^\infty} + 2 C_1 \right ),
\end{align*}
where $\underline{\mu} = \min_{\alpha \in \cA} \mu_\alpha$. From \cref{prop:gradient_bound} we deduce that $\partial u_k$ is bounded in $\sob{1}{1}$, uniformly in $k$.

Let $x \in \Gamma$. Since $\Gamma$ is connected and because each $v_k$ is continuous, there exist $N \in \NN$, $(\varv_1, \dots \varv_{N}) \in \cV^N$ and $(\alpha_1, \dots, \alpha_{N+1}) \in \cA^{N+1}$, such that
\[
 \begin{cases}
  x_k \in \Gamma_{\alpha_1}, \\
  x \in \Gamma_{\alpha_{N+1}}, \\
  \varv_n \in \Gamma_{\alpha_n} \cap \Gamma_{\alpha_{n+1}} \textnormal{ for all } 1 \leq n \leq N,
 \end{cases}
\]
and, also using \cref{prop:gradient_bound},
\begin{align*}
 \module{v_k(x) - v_k(x_k)} & \leq \int_{\pi_{\alpha_1}^{-1}(x_k)}^{\pi_{\alpha_1}^{-1}(\varv_1)} \module{\partial v_k(y)}\, dy + \left [ \sum_{n=2}^N \int_{\pi_{\alpha_n}^{-1}(\varv_{n-1})}^{\pi_{\alpha_n}^{-1}(\varv_{n})} \module{\partial v_k(y)}\, dy \right ] \\
 &+ \int_{\pi_{\alpha_{N+1}}^{-1}(\varv_N)}^{\pi_{\alpha_{N+1}}^{-1}(x)} \module{\partial v_k(y)}\, dy  \leq \norm{\partial v_k}{L^1}   \leq C_2 \leb(\Gamma)^{(q-1)/q}.
\end{align*}
It follows that
\[
 \norm{u_k}{L^\infty} \leq C_2 \leb(\Gamma)^{(q-1)/q} \quad \textnormal{for all } k \in \NN.
\]
This proves that the sequence $(u_k)_{k \in \NN}$ is bounded in $\sob{2}{1}$. From the continuous embedding $\sob{2}{1}  \hookrightarrow\cC^1(\Gamma)$ we deduce from a bootstrap argument that the sequence in also bounded in $\cC^2(\Gamma)$.

It follows that, up to extraction of a subsequence, we have
\begin{itemize}
 \item the sequence $(\lambda_k v_k(x_k))_{k \in \NN}$ converges to some constant $\rho \in \RR$,
 \item the sequence $(u_k)_{k \in \NN}$ converges in $\cC^{1}(\Gamma)$ to some $u$, with $u \in \csob{2}{\infty}$,
 \item the sequence $(\partial^2 u_k)_{k \in \NN}$ converges weakly in $L^p(\Gamma)$ to $\partial^2 u$, for any $1 < p < \infty$.
\end{itemize}
Passing to the limit in \eqref{eq:HJB_existence_1} we obtain that
\[
 \begin{cases}
 \sum_{\alpha \in \cA_\varv} \mu_\alpha \gamma_{\varv, \alpha} \partial_\alpha u(\varv) = \eta_\varv \left (\theta_\varv - \rho \right) \quad & \textnormal{for all } \varv \in \cV \setminus \partial \cV, \\
 \partial_\alpha u(\varv) = 0 \quad & \textnormal{for all } \varv \in \partial \cV.
 \end{cases}
\]
In addition, if we fix $\alpha \in \cA$ and consider $\varphi \in \cC^{\infty}(\Gamma)$, compactly supported in the interior of $\Gamma_\alpha$, we obtain
\[
 \int_\Gamma \left (- \mu \partial^2 u + H(x,\partial u) \right ) \varphi\, dx = \int_\Gamma \left (F - \rho \right ) \varphi\, dx.
\]
It follows that
\[
 - \mu \partial^2 u + H(x,\partial u) = F - \rho \quad \textnormal{almost everywhere in } \Gamma \setminus \cV.
\]
We then deduce that $u \in\cC^{2,\varsigma}(\Gamma)$ and is a classical solution to \eqref{eq:HJB}.

\underline{Proof of Uniqueness} : Let $(u_1, \rho_1)$ and $(u_2,\rho_2)$ be two solutions to \eqref{eq:HJB}. Set $w = u_1 - u_2$ and $\bar \rho = \rho_1 - \rho_2$. Up to relabelling, we may assume that $\bar \rho \geq 0$. Then $(w,\bar \rho)$ satisfies
\begin{equation}\label{eq:HJB_uniqueness_1}
 \begin{cases}
  - \mu_\alpha \partial^2 w_\alpha + b_\alpha \partial w_\alpha = - \bar \rho \leq 0 \quad & \textnormal{in } (0,L_\alpha) \textnormal{ for all } \alpha \in \cA, \\
  \eta_\varv \bar \rho + \sum_{\alpha \in \cA_\varv} \mu_\alpha \gamma_{\varv, \alpha} \partial_\alpha w(\varv) = 0 \quad & \textnormal{for all } \varv \in \cV \setminus \partial \cV,\\
  \partial_\alpha w(\varv) = 0 \quad &  \textnormal{for all } \varv \in \partial \cV, 
  \end{cases}
 \end{equation}
 where
 \[
  b_\alpha(x) = \int_0^1 \partial_p H_\alpha(x, t u_\alpha(x) + (1 - t) v_\alpha(x)) \, dt \quad \textnormal{for all } x \in (0, L_\alpha).
 \]
 We first claim that $\bar \rho = 0$. Indeed, since $w$ is continuous and $\Gamma$ compact, we may choose $x_0 \in \argmax_{\Gamma} w$. Assume first that $x_0$ belongs to the interior of $\Gamma_\alpha$ for some $\alpha \in \cA$.  Using the optimality conditions at $x_0$ in \eqref{eq:HJB_uniqueness_1} we obtain $\bar \rho \leq 0$. Assume now that $x_0 = \varv \in \cV$. From Hopf's lemma we know that we must have $\varv \in \cV \setminus \partial \cV$. In this case we have $\partial_\alpha w(\varv) \geq 0$ for all $\alpha \in \cA_\varv$, so that the junction condition in \eqref{eq:HJB_uniqueness_1} also yields $\bar \rho \leq 0$. This proves the claim.

 Now, from the strong maximum principle, we know that $w$ cannot attain a maximum point outside $\cV$ unless it is constant. On the other hand, as a consequence of Hopf's lemma, if $\varv \in \cV$ is a strict maximum of $w$ we must have $\partial_\alpha w(\varv) > 0$ for all $\alpha \in \cA$, which contradicts the junction condition in \eqref{eq:HJB_uniqueness_1}. This concludes the proof.
\end{proof}
%%%%%%%%%%%%%%%%%%%%%%%%%%%%%%%%%%%%%%%%%
 
\subsection{Verification Theorem} 
Following \cite[Section III.8]{FS2006}, we introduce an optimal control problem for a sticky diffusion process on the network and  we prove a verification theorem for the corresponding value function in term of the solution to \eqref{eq:HJB_disc}.  Our goal is to provide a justification for the junction condition considered in \eqref{eq:HJB_disc} and \eqref{eq:HJB}. In order to avoid technicalities, we stick to bounded controls.  We consider
$\mu_\alpha$, $\gamma_{\varv,\alpha}$, $\eta_\varv$ as in \ref{h:HJB1}, the set of controls
\[\gothA=\{a:\Gamma\to \RR:\, a\in\mathcal{PC}(\Gamma),\, \norm{a}{\mathcal{PC}(\Gamma)}\le R\}, \]
for some positive constant $R$, and a function $b \colon \Gamma \times \RR \to \RR$ such that, for every $\alpha \in \cA$, the function $b_{|\Gamma_\alpha \times \RR}$ in continuous for the topology of $\Gamma_\alpha \times \RR$ induced by $\Gamma \times \RR$. In this case, and for every $a \in \gothA$, we have that $b(\cdot, a(\cdot))$ also belongs to $PC(\Gamma)$ and we may therefore consider the generator of a sticky Markov process
\begin{equation}\label{eq:s1_eq1_ctr}
    \mathcal{G}^a_\alpha f(x)=\mu_\alpha \partial^2 f(x)+b_\alpha(x,a_\alpha(x))\partial f(x)\qquad \textnormal{for all } x \in \Gamma_\alpha \setminus \cV,\, \alpha \in \cA, 
\end{equation}
with domain
\begin{equation}\label{eq:s1_eq1a_ctr}
D(\cG^a) = \left \{f \in \cC^2(\Gamma) : \, \begin{array}{l} \cG^a f\in \cC(\Gamma),\,\eta_\varv \cG^a f(\varv) + \sum_{\alpha \in \cA_\varv} \mu_\alpha \gamma_{\varv,\alpha} \partial_\alpha f(\varv) = 0\\[4pt]
\textnormal{ for all } \varv \in \cV \setminus \partial \cV,\,
\partial_\alpha f(\varv) = 0\,   \textnormal{ for all } \varv \in \partial \cV \end{array} \right \}.
\end{equation}
For every $x \in \Gamma$,  let $X^{a,x} = X^a$ be the   strong Markov process corresponding to $\cG^a$ with initial distribution $\delta_x$.  Given the cost function $\ell \colon \Gamma \times \RR \to \RR$,  
%\textcolor{red}{which we assume to be measurable and   bounded}, 
continuous in   $\Gamma_\alpha \times \RR$ for $\alpha\in\cA$, $F\in PC (\Gamma )$, and $\theta_\varv\in\RR$ for $\varv\in\cV\setminus\partial\cV$, define the cost functional $J \colon \Gamma \times \gothA \to \RR$  
\begin{align*}
J(x,a) = \EE_x \Bigg [ \int_0^{+ \infty} & e^{-\lambda s}\Big [\big(\ell(X^a(s), a(X^a(s))) + F(X^a(s)) \big) \mathbbm{1}_{\{X(s) \notin \cV \}}\\
& + \sum_{\varv \in \cV \setminus \partial \cV} \theta_\varv \mathbbm{1}_{\{X^a(s) = \varv \}} \Big ]\, ds \Bigg ],
\end{align*}
where $\lambda > 0$.
This defines an infinite horizon optimal control problem where the running cost is given by $\ell + F$ inside the edges and $\theta_\varv$ when the agent is at $\varv$. Notice the vertex part of the cost is related to the stickiness parameter $\eta_\varv$ through
\[
 \EE_x \left [ \theta_\varv  \int_0^{+ \infty} \mathbbm{1}_{\{X^a(s) = \varv \}}e^{-\lambda s} \, ds \right ] =  \theta_\varv \eta_\varv \EE_x \left [ \int_0^{+\infty} e^{- \lambda s} d L^X(s,\varv) \right ],
\]
where we used \cref{thm:ito_formula}. In particular, if $\eta_\varv = 0$, the contribution of $\theta_\varv$ to $J(x,a)$ is null.
For each $x \in \Gamma$, consider the value function
\begin{equation}\label{eq:Hamiltonian}
 V(x) = \inf_{a \in \gothA} J(x,a).
\end{equation}
Our goal is to identify $V$ as the unique solution to the following discounted Hamilton-Jacobi-Bellman equation
\begin{equation}\label{eq:HJB_disc_verif}
	\begin{cases}
		- \mu_\alpha  \partial^2 u(x) + H(x,\partial u(x)) +\lambda u(x) = F(x) \quad & \textnormal{for all } x \in \Gamma_\alpha \setminus \cV,\, \alpha \in \cA, \\
		u_{|\Gamma_\alpha}(\varv) = u_{|\Gamma_\beta}(\varv) \quad & \textnormal{for all } \alpha, \beta \in \cA_\varv, \, \varv \in  \cV, \\
		  \sum_{\alpha \in \cA_\varv}   \mu_\alpha \gamma_{\varv, \alpha} \partial_\alpha u(\varv) =\eta_\varv (  \theta_\varv -\lambda u(\varv) ) \quad & \textnormal{for all } \varv \in \cV \setminus \partial \cV, \\
		\partial_\alpha u(\varv) = 0 \quad & \textnormal{for all } \varv \in \partial \cV,
	\end{cases}
\end{equation}
where
\[
 H_\alpha(x,p) := \sup_{|a|\le R} \left \{- b_\alpha(x,a)p - \ell_\alpha(x,a) \right \},\qquad \alpha\in \cA.
\]

\begin{thm}\label{thm:verification_1}
 Let $u \in \cC^2(\Gamma)$ be a solution to \eqref{eq:HJB_disc_verif}. Then, for every $a \in \gothA$, we have $u(x) \leq J(x,a)$. In particular we have $u(x) \leq V(x)$.
\end{thm}

\begin{proof}
 Let $u \in \cC^2(\Gamma)$ be a solution to \eqref{eq:HJB_disc_verif} and let $a \in \gothA$. Consider $X^a$ the Feller process generated by \eqref{eq:s1_eq1_ctr}, \eqref{eq:s1_eq1a_ctr}. We define a function $f \colon \RR_+ \times \Gamma \to \RR$ by setting $f(t,x) = u(x)e^{-\lambda t}$.
 Let also $(t_n)_{n\in \NN}$ be a sequence in $\RR_+$ converging to $+ \infty$. From \cref{thm:ito_formula} we deduce that, for every $n \in \NN$,
 \begin{equation}\label{eq:verification_1}
 \begin{split}
  \EE_x&[f(t_n, X^a(t_n))]   = f(0,x) + \EE_x \left [ \int_0^{t_n} \left ( \partial_s + \cG^a \right) f(s,X^a(s)) \mathbbm{1}_{\{X(s) \notin \cV\}}\, ds \right ] \\
  &  + \sum_{\varv \in \cV \setminus \partial \cV} \EE_x \Big [ \int_0^{t_n} \Big ( \eta_\varv \partial_s f(s,\varv) -  \sum_{\alpha \in \cA_\varv} \mu_\alpha \gamma_{\varv,\alpha} \partial_\alpha f(s,\varv) \Big ) d L^{X^a}(s,\varv) \Big].
  \end{split}
 \end{equation}
 Notice that, for $s > 0$ and $X^a(s) \notin \cV$,
 \begin{align*}
   ( \partial_s& + \cG^a ) f(s,X^a(s))  = \left[- \lambda u(X^a(s)) + \mu \partial^2 u(X^a(s)) + b(X^a(s), a(X^a(s))) \partial u(X^a(s)) \right] e^{-\lambda s} \\
   & \geq \left[- \lambda u(X^a(s)) + \mu \partial^2 u(X^a(s)) - H(X^a(s), \partial u(X^a(s))) - \ell(X^a(s), a(X^a(s))) \right] e^{-\lambda s} \\
   & = - \left [F(X^a(s)) + \ell(X^a(s), a(X^a(s))) \right] e^{-\lambda s}.
 \end{align*}
 From the junction condition in \eqref{eq:HJB_disc_verif} we also have
 \begin{align*}
  \eta_\varv \partial_s f(s,\varv) - \sum_{\alpha \in \cA_\varv} \mu_\alpha \gamma_{\varv,\alpha} \partial_\alpha f(s,\varv) & = \Big (-\eta_\varv \lambda u(\varv) - \sum_{\alpha \in \cA_\varv} \mu_\alpha \gamma_{\varv,\alpha} \partial_\alpha u(\varv) \Big ) e^{-\lambda s}  = - \eta_\varv \theta_\varv e^{- \lambda s}.
 \end{align*}
 So that \eqref{eq:verification_1} and \cref{thm:ito_formula} imply
 \begin{align*}
  \EE_x \left [ u(X^a(t_n)) e^{-\lambda t_n} \right] & \geq u(x) - \EE_x \left [ \int_0^{t_n} \left (F(X^a(s)) + \ell(X^a(s), a(X^a(s))) \right) \bm{1}_{\{X^a(s) \notin \cV\}} e^{-\lambda s}\, ds \right ] \\
  &  - \sum_{\varv \in \cV \setminus \partial \cV} \EE_x \left [ \int_0^{t_n} \eta_\varv  \theta_\varv e^{- \lambda s} d L^{X^a}(s,\varv) \right] \\
  & = u(x) - \EE_x \left [ \int_0^{t_n} \left (F(X^a(s)) + \ell(X^a(s), a(X^a(s))) \right) \bm{1}_{\{X(s) \notin \cV\}} e^{-\lambda s}\, ds \right ] \\
  &  - \sum_{\varv \in \cV \setminus \partial \cV} \EE_x \left [ \int_0^{t_n}  \theta_\varv \mathbbm{1}_{\{ X^a(s) = \varv \}} e^{- \lambda s} ds \right].
 \end{align*}
Since $u$, $F$ and $\ell(\cdot,a(\cdot))$ are bounded, we may apply Lebesgue's convergence theorem to pass to the limit $n \to \infty$ and obtain $u(x) \leq J(x,a)$.
\end{proof}

To get the reverse inequality, we consider the following assumption
\begin{enumerate}[label={\bf (H\arabic*)},resume ]
	\item \label{h:HJB4} 
For each $(x,p)	\in \Gamma_\alpha\times \RR$, there exists a unique  $a^\star_\alpha=a^\star_\alpha(x,p)$,   with $ \module{\alpha^\star_\alpha} \leq R$ for every $\alpha \in \cA$, such that
\[
H_\alpha(x,p) = -b_\alpha(x,a^\star_\alpha)p - \ell_\alpha(x,a^\star_\alpha)
\]	
and the map  $a^\star_\alpha \colon \Gamma_\alpha \times \RR \to \RR$ is continuous.
\end{enumerate}
%%%
\begin{prop}
Assume \ref{h:HJB4} and let $u \in \cC^2(\Gamma)$ be a solution to \eqref{eq:HJB_disc_verif}. Then $u(x) \geq V(x)$ for all $x \in \Gamma$.
\end{prop}
\begin{proof}
Using \ref{h:HJB4}, we have
\[
H(x,\partial u(x)) = \partial_p H(x, \partial u(x)) \partial u(x) - \ell(x,a^\star(x)).
\]
We may consider the generator $\cG^\star$ defined as in \eqref{eq:s1_eq1_ctr} with $b$ replaced by $\partial_p H(\cdot, \partial u(\cdot))$ and the corresponding Markov process $X^\star$. 
 Then the same computation as in \cref{thm:verification_1} yields
 \begin{align*}
  u(x) & = \EE_x \Bigg [ \int_0^{+\infty} \left (F(X^\star(s)) + \ell(X^\star(s), a^\star(X^\star(s))) \right) \bm{1}_{\{X^a(s) \notin \cV\}} e^{-\lambda s}\, ds  \Bigg ] \\
  & \qquad +  \sum_{\varv \in \cV \setminus \partial \cV} \EE_x \Bigg [ \int_0^{+ \infty}  \theta_\varv \mathbbm{1}_{\{ X^{a^\star}(s) = \varv \}} e^{- \lambda s} ds \Bigg] \geq V(x)
 \end{align*}
\end{proof}
%%%%%%%%%%%%%%%%%%%%
%                  %
%%%%%%%%%%%%%%%%%%%%

\section{The MFG system}\label{sec:5}
Aim of this section is to study the Mean Field Games model with sticky transition condition at the vertices. We consider the following MFG system 
\begin{equation}\label{eq:MFG}
\left\{\begin{aligned}
	(i)\,
	&- \mu_\alpha  \partial^2 u(x) + H(x,\partial u(x)) + \rho = F[\mf]  & \textnormal{for all } x \in \Gamma_\alpha \setminus \cV ,\, \alpha \in \cA, \\
	&u_{|\Gamma_\alpha}(\varv) = u_{|\Gamma_\beta}(\varv)   & \textnormal{for all } \alpha, \beta \in \cA_\varv, \, \varv \in  \cV, \\
	&\sum_{\alpha \in \cA_\varv} \mu_\alpha \gamma_{\varv, \alpha} \partial_\alpha u(\varv) = \eta_\varv \left ( \theta_\varv + F[\mf](\varv) - \rho \right)  & \textnormal{for all } \varv \in \cV \setminus \partial \cV, \\
	&\partial_\alpha u(\varv) = 0   & \textnormal{for all } \varv \in \partial \cV, \\
	&\int_\Gamma u \, dx = 0 \\	
	%%%%%%
	(ii)\, &\,\mf = m \leb + \sum_{\varv \in \cV \setminus \partial \cV} \eta_\varv T_\varv[m] \delta_{\varv},\\
	- &\mu_\alpha \partial^2 m(x) -  \partial \left(\partial_p H(x, \partial u(x)) m(x) \right ) = 0   & \textnormal{for all } x \in \Gamma_\alpha,\, \alpha \in \cA, \\
	&\frac{m_{|\Gamma_\alpha}(\varv)}{\gamma_{\varv,\alpha}} = \frac{m_{|\Gamma_\beta}(\varv)}{\gamma_{\varv,\beta}} =: T_\varv[m]   & \textnormal{for all } \alpha, \beta \in \cA_\varv, \, \varv \in \cV \setminus \partial \cV, \\
	&\sum_{\alpha \in \cA_\varv} \mu_\alpha \partial_\alpha m_{|\Gamma_\alpha}(\varv) + n_{\varv,\alpha}  m_{|\Gamma_\alpha}(\varv)  \partial_p H_\alpha(\varv, \partial u_{|\Gamma_\alpha}(\varv)) = 0   & \forall \varv \in \cV, \\
	&m \geq 0,  \, 1 \geq \int_\Gamma m\, dx = 1 -  \sum_{\varv \in \cV \setminus{\partial \cV}} \eta_\varv T_\varv[m] \geq 0.
\end{aligned}\right.
\end{equation}
In \eqref{eq:MFG}, the item \textit{(i)} corresponds to the Hamilton-Jacobi-Bellman equation,
\textit{(ii)} to the Fokker-Planck equation. Recall that this system arises from a Mean Field Game with a long-term average cost \cite{LL2007}, where the dynamics of a typical player are modeled by a sticky diffusion on $ \Gamma$. In this framework, the dynamics behaves as a standard diffusion within each edge with the agent controlling the velocity. Upon reaching a vertex $\varv$, the agent enters the edge $\alpha \in \cA_\varv$ with probability $\mu_\alpha \gamma_{\varv, \alpha}$ and is subject to stickiness, the latter being quantified by $\eta_\varv$. The mapping $F$, called the coupling, encodes the interactions between players through their distribution and $\theta_\varv$ is a fixed cost that the player must pay when at vertex $\varv$. We refer the reader to \cite{CM2016,ADLT2019} for a formal derivation of the system in the nonsticky setting.

Note that the coupling and the ergodic constant also appear in the right hand side of the Kirchhoff condition, in addition to the equation. We observe that, given the  solution $m$ of the Fokker-Planck equation inside the edges with the corresponding transition conditions at the vertices, the measure $\mf$ is completely determined by the coefficients $T_\varv[m]$. Hence we consider as unknows in \eqref{eq:MFG} $u$, $\rho$ and $m$.
%%%%%%%%%%%%%%%%%%%%%%%%%%%%%%%%%%%%%
\begin{thm}[Existence of solutions]
  Assume \ref{h:HJB1}, \ref{h:HJB2} and that $F \colon\cP_1(\Gamma) \to PC^\varsigma(\Gamma)$ is continuous and takes values in a bounded subset of $PC^{\varsigma}(\Gamma)$. Then there exists a solution $(u,\rho, m ) \in \cC^{2,\varsigma}(\Gamma)\times \RR  \times \bssob$ to the MFG system \eqref{eq:MFG}.
\end{thm}

\begin{proof}
 The proof  follows the  by now standard argument for existence in second order MFG and relies on Schauder's fixed point theorem.
 We define
\begin{equation*}
	\cM := \left \{ m \in PC(\Gamma) : \begin{array}{l} m \geq 0,\, \frac{m_{|\Gamma_\alpha}(\varv)}{\gamma_{\varv,\alpha}} = \frac{m_{|\Gamma_\beta}(\varv)}{\gamma_{\varv,\beta}} =: T_\varv[m] \quad\textnormal{for all } \alpha, \beta \in \cA_\varv,\\[4pt]
	\varv \in \cV \setminus \partial \cV\quad \textnormal{and } \int_\Gamma m \, dx  + \sum_{\varv \in \cV \setminus{\partial \cV}} \eta_\varv T_\varv[m] = 1 \end{array} \right \}.
\end{equation*}
Notice that the  previous set  is a closed convex subset of $PC(\Gamma)$.
 Let $m \in \cM$ and define $\mf \in \cP_1(\Gamma)$ according to \eqref{eq:mu}. Since $F[\mf] \in PC^{\varsigma}(\Gamma)$ we know from \cref{thm:HJB_well_posed} that there exists a unique solution $(u,\rho) \in \cC^{2,\varsigma}(\Gamma) \times \RR$ to \eqref{eq:HJB} with $F$ replaced by $ F[\mf]$ and $\theta_\varv$ replaced by $ \theta_\varv + F[\mf](\varv)$. From \cref{prop:FP_well_posed} we then know that there exists a unique weak solution $\bar m \in \bssob \subset PC(\Gamma)$ to \eqref{eq:FP}. Since $\bar m \in \cM$, we can define a mapping $\Phi \colon \cM \to \cM$ by setting $\Phi(m) = \bar m$. We claim that $\Phi$ satisfies the assumptions of Schauder's fixed point theorem \cite[Corollary 11.2]{GT2001}.

 We first prove that $\Phi$ is continuous. Let $(m^n)_{n \in \NN}$ be a sequence in $\cM$ and converging to $m \in \cM$. Define $\mf^n$ and $\mf$ according to \eqref{eq:mu}. From \cref{prop:continuity_density} we know that $\mf^n$ converges to $\mf$ in $\mathcal{P}_1(\Gamma)$. Since $F $ is   continuous, we deduce that $F[\mf^n]$ converges to $F[\mf]$ in $PC^{\varsigma}(\Gamma)$. For each $n \in \NN$, from \cref{thm:HJB_well_posed}, there exists a unique solution $(u^n,\rho^n) \in \cC^{2,\varsigma}(\Gamma) \times \RR$ to \eqref{eq:HJB} with $\int u\, dx = 0$,  $F$ replaced by $F[\mf^n]$ and $\theta_\varv$ by $\theta_\varv + F[\mf^n](\varv)$.\par
 We claim that the sequence $(\rho^n)_{n \in \NN}$ is bounded. Indeed, integrating the equation satisfied by $(u^n, \rho^n)$, we have
 \begin{align*}
  \module{\rho^n} \leb(\Gamma) = \int_{\Gamma} \sgn(\rho^n) \rho^n\, dx  \leq \int_\Gamma \mu \module{\partial^2 u^n} + C_H(1 + \module{\partial u^n}^q) + \module{F[\mf^n]}\, dx.
 \end{align*}
 From \cref{thm:HJB_well_posed} we know that $u^n$ is bounded in $\sob{2}{1}$ and that $\partial u^n$ is bounded in $\Lx{q}$, so that the right-hand side in the last inequality is bounded. This proves the claim.\par
 By bootstrapping the regularity of $u^n$ we see that $(u^n)_{n \in \NN}$ is bounded in $\cC^{2,\varsigma}(\Gamma)$. We may extract a subsequence converging to a solution $(u,\rho) \in \cC^2(\Gamma) \times \RR$ to \eqref{eq:HJB} with $\int u \, dx = 0$, $F$ replaced by $F[\mf]$ and $\theta_\varv$ by $\theta_\varv + F[\mf](\varv)$. Since this solution is unique we conclude that the whole sequence must converge to $(u,\rho)$. Then $\bar m^n := \Phi(m^n)$ is the unique solution to \eqref{eq:FP} with $b$ replaced by $\partial_p H(\cdot, \partial u^n)$. Since $\partial u_{|\Gamma_\alpha}^n$ converges uniformly to $\partial u_{\Gamma_\alpha}$ and $\partial_p H_\alpha$ is continuous for every $\alpha \in \cA$, we may apply \cref{prop:FP_stability} to conclude that $\bar m^n$ converges to $\bar m = \Phi(m)$ in $\cM$. This proves the continuity of $\Phi$.

 The compactness of $\Phi$ follows from the fact that $F$ takes values in a bounded subset of $PC^{\varsigma}(\Gamma)$ and the compact embedding $\bssob \hookrightarrow PC(\Gamma)$. We can therefore apply Schauder's fixed point theorem to conclude the proof.
\end{proof}
We now briefly discuss uniqueness of the solution to \eqref{eq:MFG} under monotonicity assumptions on the coupling cost.
\begin{thm}\label{thm:uniqueness}
 For all for all $\mf_1$, $\mf_2\in\cP_1(\Gamma)$ with $\mf_i = m_i \leb + \sum_{\varv \in \cV \setminus \partial \cV} \eta_\varv T_\varv[m_i] \delta_{\varv}$, assume one of the following assumptions:
  \begin{enumerate}[label=(\roman*)]
   \item the map $p \mapsto H(x,p)$ is convex and $F$ is strictly monotone, in the sense that
   \[
    \int_\Gamma (F[\mf_1] - F[\mf_2])(\mf_1 - \mf_2)(dx) > 0\quad \textnormal{};
   \]
   \item the map $p \mapsto H(x,p)$ is  strictly convex and $F$ is monotone, in the sense that
   \[
    \int_\Gamma (F[\mf_1] - F[\mf_2])(\mf_1 - \mf_2)(dx) \geq 0\quad \textnormal{for all $\mf_1$, $\mf_2\in\cP_1$}.
   \]
  \end{enumerate}
   Then there exists at most one solution $(u,\rho, m ) \in \cC^{2,\varsigma}(\Gamma)\times \RR  \times     \bssob$ to the MFG system \eqref{eq:MFG}.

\end{thm}

\begin{proof}
 Let $(u_1,m_1, \rho_1)$ and $(u_2, m_2, \rho_2)$ be two solutions to \eqref{eq:MFG} and
 define $\mf_i$, $i=1,2$, as in \eqref{eq:mu}. Set $\bar u = u_1 - u_2$,  $\bar \mf = \mf_1 - \mf_2$, $\bar m = m_1 - m_2$ and $\bar \rho = \rho_1 - \rho_2$. \par

 Integrating the equation satisfied by $\bar u$ against $\bar m$ and using the transition condition for $u_1$ and $u_2$, we obtain
 \begin{equation*}
 \begin{split}
  & \int_\Gamma F[\mf_1] - F[\mf_2] \bar m dx = \int_\Gamma \left (- \mu \partial^2 \bar u + H(x,\partial u_1) - H(x, \partial u_2) + \bar \rho \right) \bar m\, dx \\
  & \qquad = \int_\Gamma \mu \partial \bar u \partial \bar m +  \left ( H(x,\partial u_1) - H(x, \partial u_2) \right) \bar m \, dx \\ & \qquad \quad  - \left [ \sum_{\varv \in \cV \setminus \partial \cV} T_\varv[\bar m] \sum_{\alpha \in \cA_\varv} n_{\varv,\alpha} \mu_\alpha \gamma_{\varv, \alpha} \partial_\alpha \bar u(\varv) \right ] + \int_\Gamma \bar \rho \bar m\, dx \\
  & \qquad = \int_\Gamma \mu \partial \bar u \partial \bar m +  \left ( H(x,\partial u_1) - H(x, \partial u_2) \right) \bar m \, dx \\
  & \qquad \quad + \left [\sum_{\varv \in \cV \setminus \partial \cV} \eta_\varv T_\varv[\bar m] \left (\bar \rho - (F[\mf_1](\varv) - F[\mf_2](\varv) \right) \right ] + \int_\Gamma \bar \rho \bar m\, dx \\
  & \qquad = \int_\Gamma \mu \partial \bar u \partial \bar m +  \left ( H(x,\partial u_1) - H(x, \partial u_2) \right) \bar m \, dx  \\
  & \qquad \quad - \left [\sum_{\varv \in \cV \setminus \partial \cV} \eta_\varv T_\varv[\bar m] \left ((F[\mf_1](\varv) - F[\mf_2](\varv) \right) \right ] + \int_\Gamma \bar \rho \bar \mf(dx).
 \end{split}
 \end{equation*}

 Using the fact that $\int_\Gamma \bar \rho \bar \mf(dx) = 0$, we conclude that
 \begin{equation}\label{eq:uniq1}
  \int_\Gamma (F[\mf_1] - F[\mf_2]) \bar \mf (dx) = \int_\Gamma \mu \partial \bar u \partial \bar m +  \left ( H(x,\partial u_1) - H(x, \partial u_2) \right) \bar m \, dx.
 \end{equation}

 On the other hand, using $\bar u$ as a test-function in the equation satisfied by $\bar m$ we obtain
 \begin{equation}\label{eq:uniq2} 	 
  \int_\Gamma \mu \partial \bar m \partial \bar u + \left (m_1 \partial_p H(x,\partial u_1) - m_2 \partial_p H(x,\partial u_2) \right) \partial \bar u = 0.
 \end{equation}
 
 The rest of the proof follows the usual argument introduced in \cite{LL2007}. Indeed, subtracting \eqref{eq:uniq2} from \eqref{eq:uniq1}, we get
 \begin{align*}
 	0&=\int_\Gamma (F[\mf_1] - F[\mf_2]) \bar \mf (dx)\\
 	&+\int_\Gamma m_1[H(x,\partial u_2)-H(x,\partial u_1)-\partial_p H(x,\partial u_1)\partial \bar u]dx\\
 	&+\int_\Gamma m_2[H(x,\partial u_1)-H(x,\partial u_2)-\partial_p H(x,\partial u_2)\partial \bar u]dx.
 \end{align*}
Assuming that $F$ is strictly monotone, the first integral is non negative. By convexity of $H$ and positivity of $m_i$, the other two integrals are non negative. It follows that 
 $\mf_1=\mf_2$ and also $u_1=u_2$. Recalling  \eqref{eq:mu}, we also get $m_1=m_2$. If $F$ is   monotone, we conclude in a similar way.
\end{proof}
%%%%%%%
%     %
%%%%%%%
 \begin{ex}
 	We give an example of coupling cost satisfying the assumptions of Theorem \ref{thm:uniqueness}. Given   $\mf=m \leb + \sum_{\varv \in \cV \setminus \partial \cV} \eta_\varv T_\varv[m] \delta_{\varv}$  we define
 	\begin{equation*}
 		F[\mf](x)=
 		\begin{cases}
 			F_I[m](x)\qquad& x\in\Gamma\setminus \cV,\\[4pt]
 			F_V(T_\varv[m])  &x \in  \cV\setminus\partial\cV,
 		\end{cases}
 	\end{equation*}
 	where
\[\int_\Gamma (F_I[m_1] - F_I[m_2])(m_1 - m_2)(dx) > 0\] 	
for $m_i\in  L^1(\Gamma)$, $m_i\ge 0$, $i=1,2$, and $F_V:\RR^+\to\RR$ is an increasing function.
 \end{ex}

\printbibliography

\end{document}